\documentclass[a4paper]{amsart}
\usepackage{amssymb}
\usepackage{mathrsfs}
\newtheorem{theorem}{Theorem}[section]
\newtheorem{lem}[theorem]{Lemma}
\newtheorem{cor}[theorem]{Corollary}
\newtheorem{prop}[theorem]{Proposition}
\theoremstyle{definition}
\newtheorem{defi}[theorem]{Definition}
\newtheorem{example}[theorem]{Example}

\theoremstyle{remark}
\newtheorem{rem}[theorem]{Remark}

\newcommand{\BC}{\mathbb{C}}            %% field of complex numbers
\newcommand{\BZ}{\mathbb{Z}}             %% ring of natural numbers  
\newcommand{\BM}{\mathbb{M}}           %% field of meromorphic functions
\newcommand{\BI}{\mathbf{i}}                 %% imaginary number
\newcommand{\Hlie}{\mathfrak{h}}          %% Cartan subalgebra
\newcommand{\End}{\mathrm{End}}       %% endomorphism ring
\newcommand{\CM}{\mathcal{M}}                %% affine monoid

\newcommand{\BGG}{\mathcal{O}}      %% category BGG
\newcommand{\Simple}{\mathcal{S}}          %% simple objects
\newcommand{\qc}{\chi_{\mathrm{q}}}            %% q-character
\newcommand{\qcy}{\chi^{\mathrm{Y}}_{\mathrm{q}}}            %% q-character Yangian
\newcommand{\CF}{\mathcal{F}}          %% cateogry of finite-dimensional vector spaces
\newcommand{\CFm}{\mathcal{F}_{\mathrm{mer}}}    %% meromorphic subcategory

\newcommand{\CV}{\mathcal{V}}                  %% category of weighted vector spaces
\newcommand{\wt}{\mathrm{wt}}          %% the set of weights
\newcommand{\wtimes}{\bar{\otimes}}     %% weighted tensor products
\newcommand{\dt}{\widetilde{\otimes}}     %% dynamical tensor product
\newcommand{\mol}{\mu_{\mathbf{l}}}  %% left moment map
\newcommand{\mor}{\mu_{\mathbf{r}}}  %% right moment map

\newcommand{\BV}{\mathbf{V}}             %% vector representation
\newcommand{\CW}{\mathcal{W}}        %% asymptotic representations
\newcommand{\BW}{\mathbf{W}}          %% prefundamental representation

\newcommand{\CE}{\mathcal{E}}           %% elliptic quantum group
\newcommand{\BR}{\mathbf{R}}            %% R matrix

\newcommand{\BY}{\mathbf{Y}}            %% Yangian

\begin{document}
\title{Baxter operators and asymptotic representations}
\author{Giovanni Felder}
\address{G.F: Departement Mathematik,
ETH Z\"{u}rich, 8092 Z\"{u}rich, Switzerland}
\email{giovanni.felder@math.ethz.ch}
\author{Huafeng Zhang}
\address{H.Z: Departement Mathematik,
ETH Z\"{u}rich, 8092 Z\"{u}rich, Switzerland \& 
Institut f\"ur Theoretische Physik, 
ETH Z\"{u}rich, 8093 Z\"{u}rich, Switzerland }
\email{huafeng.zhang@math.ethz.ch}
\begin{abstract}
We introduce a category $\mathcal O$ of representations of the
elliptic quantum group associated with $\mathfrak{sl}_2$ with
well-behaved $q$-character theory. We derive separation of variables
relations for asymptotic representations in the Grothendieck ring of
this category.  Baxter $Q$-operators are obtained as transfer matrices
for asymptotic representations and obey $TQ$-relations as a
consequence of the relations in $K_0(\mathcal O)$.
\end{abstract}
\maketitle
\section{Introduction}\label{s-intro}

In Baxter's formulation \cite{Baxter72} the problem of solving exactly
solvable models of statistical mechanics reduces to finding the common
spectrum of a family $T(z)$ of commuting endomorphisms of a vector
space depending on a complex spectral parameter $z$. For example in
the six- or eight-vertex model the vector space is a tensor product of
copies of $\mathbb C^2$. The Bethe ansatz is a celebrated method to
compute eigenvectors and eigenvalues of the transfer matrices
$T(z)$. In this method one seeks eigenvectors among a specific family
of vectors depending on parameters, called Bethe roots. The condition for the parameters
to yield eigenvectors is a system of algebraic equations---the Bethe
ansatz equations---for the Bethe roots, and one gets an implicit
description of the eigenvectors and eigenvalues. This was done by Lieb
\cite{Lieb} for the ice model, a special case of the six-vertex model,
adapting the method of Bethe \cite{Bethe1931}, who had considered the
Heisenberg spin chain.  Baxter \cite{Baxter72} devised another method
to obtain the Bethe ansatz equation by directly computing the
eigenvalues. The method also works for models where the Bethe ansatz
fails, such as the eight-vertex model, but does not give direct
information on the eigenvectors. One uses a new set of commuting
endomorphisms $Q(z)$---the Q-operator---that commute with the transfer
matrices and obey a functional relation, the Baxter TQ-relation:
\[
T(z)Q(z)=\phi(z)Q(z+\hbar)+\phi(z+\hbar)Q(z-\hbar),
\] 
for some specific function $\phi$ ($\hbar$ is a parameter of the
model).  Moreover $T(z)$ and $Q(z)$ are entire functions of $z$ with
known functional behaviour: for example in the eight-vertex model
$Q(z)$ has theta function-like double periodicity properties. The
common eigenvalues obey the same equation and the functional relation
determines their form: for example in the eight-vertex model the
double periodicity property implies that eigenvalues of $Q(z)$ have
the form $\prod_{i=1}^n\theta(z-z_i)$ where $\theta$ is the odd Jacobi
theta function for some elliptic curve depending on the parameters of
the model and $z_i$ are unknowns to be determined.  Inserting $z=z_i$ in
Baxter's TQ-relation we obtain the Bethe ansatz equations
\[ \prod_{j:j\neq i} \frac{\theta(z_i-z_j+\hbar)}
  {\theta(z_i-z_j-\hbar)}=\frac{\phi(z_i+\hbar)}{\phi(z_i)},\quad i=1,\dots,n.
\] Each solution of this system of $n$ equations for $n$ unknowns
gives a candidate for an eigenvalue of $Q$, from which the
corresponding eigenvalue of $T$ can be computed from the
$TQ$-relation.  

The approach to exactly solvable models based on transfer matrices and
Bethe ansatz was extended to a variety of models of statistical
mechanics, quantum integrable models and quantum field theory. It was
developed into a full-fledged theory by the Leningrad school under the
names Quantum Inverse Scattering Method, and Algebraic Bethe Ansatz,
leading to the theory of quantum groups, see the lecture notes
\cite{Faddeev} for a review.  In modern terminology, if we have a
(suitable) pair $V,W,$ of representations of a quantum group, for
example an affine quantum enveloping algebra, we obtain an $R$-matrix
$R_{V,W}$, an endomorphism of $V\otimes W$, and transfer matrices are
traces over one of the tensor factors, say $W$, called auxiliary
space: $t_W=\mathrm{tr}_W R_{V,W}$. The basic fact, that follows from
the Yang--Baxter equation, is that varying $W$ we get commuting
endomorphisms of $V$, and the transfer matrix $T(z)$ of the six-vertex
model is obtained for $W$ a two-dimensional representation of the
quantum affine algebra of $\mathfrak{sl}_2$ with evaluation parameter $z$ and
$V$ a tensor product of two-dimensional representations.

The representation theory meaning of the Baxter $Q$-operator was
understood much later: Bazhanov, Lukyanov and Zamolodchikov
\cite{BazhanovLukyanovZamolodchikov1997,
BazhanovLukyanovZamolodchikov1999} gave a construction of a
$Q$-operator in a quantum field theory context as a transfer matrix
for a certain infinite dimensional auxiliary space.  Frenkel and
Hernandez constructed $Q$-operators for arbitrary untwisted affine
quantum enveloping algebras as transfer matrices defined as traces
over certain representations of a Borel subalgebra of the quantum loop algebra. These pre-fundamental representations belong to a category that had
previously been introduced and studied by Hernandez and Jimbo
\cite{HJ}. The Baxter TQ-relations follow then from
relations in the Grothendieck ring of this category. The point is
that the $R$-matrix is given by the action of the tensor product
of two opposite Borel subalgebras, so that $R_{V,W}$ makes sense
if the auxiliary space $W$ is a representation of a Borel subalgebra.

The goal of this paper is to extend these results to the elliptic
quantum group associated to $\mathfrak{sl}_2$. One new feature is the
appearance of the dynamical parameter, so that transfer matrices are
now difference operators acting on functions of the dynamical
parameters and it requires some care to extend the constructions to
this case. A more serious new difficulty is that the $R$-matrix does not
have the same triangular structure as in the affine case, so it does
not make sense to consider representations of Borel subalgebras. We
show however that, at least in the $\mathfrak{sl}_2$-case considered
here, one can construct $Q$-operators as transfer matrices for certain
infinite dimensional representations belonging to a suitable abelian
category $\mathcal O$. The objects of this category have well-defined
$q$-characters \cite{Kn}, \cite{FR}, and this
allows us to derive relations in the Grothendieck ring, from which
$TQ$-relations are obtained.

The paper is structured as follows. In Section 2 we review the theory
of the elliptic quantum group associated to $\mathfrak {sl}_2$,
discuss the notion of eigenvalues for difference operators, and define
asymptotic representations. A category $\mathcal O$ of representations
containing all asymptotic representations is constructed in Section
3. We define the $q$-character map and show that it is an injective
ring homomorphism from the Grothendieck ring $K_0(\mathcal O)$ to a
suitable commutative ring $\mathcal M_t$ and deduce relations among
asymptotic representations in $K_0(\mathcal O)$. We identify highest
weights of simple objects in $\mathcal O$ in Section 4 and describe
the classes of finite dimensional modules in $K_0(\mathcal O)$ in
terms of asymptotic representations. Transfer matrices associated with
representations in $\mathcal O$ are constructed in Section 5. In
particular we construct the $Q$-operator and prove its $TQ$-relations
using the relations found in Section 3. Future directions are
indicated in Section 6. The same construction also gives $Q$-operators for
the Yangian of $\mathfrak{sl}_2$ as transfer matrices associated with
asymptotic Yangian representations. We indicate this in the Appendix.
\section{The elliptic quantum group and its representations}\label{s-pre}
Let $\BM$ be the field of meromorphic functions $g(x)$ on $x \in \BC$. It contains the subfield $\BC$ of constant functions. We work mostly with $\BM$-vector spaces. An $\BM$-linear (or more generally $\BC$-linear) map $\Phi$ of two $\BM$-vector spaces will also be denoted by $\Phi(x)$ if the dependence on $x$ needs to be precised. Fix $\hbar \in \BC^{\times}$.

\subsection{Meromorphic eigenvalues.} To work with the difference operators in representations of the elliptic quantum group, let us introduce a category $\CF$. An object in $\CF$ is a triplet $(V, K_+(z), K_-(z))$ with the following conditions:
\begin{itemize}
\item[(F1)]  $V$ is a finite-dimensional $\BM$-vector space, and $K_{\pm}(z): V \longrightarrow V$ are $\BC$-linear maps with parameter $z \in \BC$ (maybe not well-defined on all of $\BC$);
\item[(F2)] if $(v_j)_{1\leq j \leq n}$ is an $\BM$-basis of $V$, then there exist meromorphic functions $a_{ij}^{\pm}(z;x)$ on $(z,x) \in \BC^2$ for $1\leq i,j\leq n$ such that
\begin{equation}  \label{def: difference operator}
K_{\pm}(z) \left(\sum_{j=1}^n g_j(x) v_j\right) = \sum_{i,j=1}^n a_{ij}^{\pm}(z;x) g_j(x\pm \hbar) v_i \quad \forall g_j(x) \in \BM.    
\end{equation}
\end{itemize}
A morphism $(V,K_+(z),K_-(z)) \longrightarrow (V',K_+'(z),K_-'(z))$ in $\CF$ is an $\BM$-linear map $\Phi: V \longrightarrow V'$ such that $\Phi K_{\pm}(z) = K_{\pm}'(z) \Phi: V \longrightarrow V'$. When there is no confusion, for simplicity we also denote $(V,K_+(z),K_-(z))$ by $V$. 

$\mathrm{Hom}_{\CF}(V,V')$ is a sub-$\BC$-vector space of $\mathrm{Hom}_{\BM}(V,V')$, making $\CF$ into an abelian category. The (co)kernel of a morphism $\Phi: V \longrightarrow V'$ in $\CF$ is induced from that of the $\BM$-linear map  $\Phi: V \longrightarrow V'$. By induction on $\dim_{\BM} (V)$, every object $V$ in $\CF$ admits a Jordan-H\"{o}lder series: a sequence of subobjects $V = V_m \supset V_{m-1} \supset \cdots \supset V_1 \supset V_0 = 0$ such that each quotient object $V_i/V_{i-1}$ is simple.

Let $a^{\pm}(x) \in \BM^{\times}$. Define $\BM[a^+(z),a^-(z)] := (\BM, K_+(z),K_-(z)) \in \CF$ by\footnote{Functions in $\BM$ have variable $x$. Be aware of the change of variables $x \mapsto z$ here.}
$$ K_{\pm}(z) (g(x)) = g(x\pm\hbar) a^{\pm}(z) \quad \mathrm{for}\ g(x) \in \BM. $$
\begin{defi} \label{def: meromorphic eigenvalue}
$\CFm$ is the full subcategory of $\CF$ consisting of objects $V$ whose quotient objects in Jordan-H\"{o}lder series are isomorphic to the $\BM[a^{+}(z),a^-(z)]$ for $a^{\pm}(x) \in \BM^{\times}$. 
\end{defi}
From the definition it is not difficult to show 
\begin{lem} \label{lem: abelian category meromorphic eigenvalue}
$\CFm$ is an abelian subcategory of $\CF$. An object $V$ of $\CF$ is in $\CFm$ if and only if there exists an $\BM$-basis $(v_i)_{1\leq i \leq n}$ of $V$ with respect to which the matrices $(a_{ij}^{\pm}(z;x))_{1\leq i,j \leq n}$ in Equation \eqref{def: difference operator} are upper triangular whose diagonals are independent of $x$ and non-zero.
\end{lem}
Notice that: $v_1 \in V$ is an eigenvector of the $\BC$-linear maps $K_{\pm}(z)$ of eigenvalue $a_{11}^{\pm}(z;x) = a_{11}^{\pm}(z)$ respectively; $(V,K_-(z)^{-1}, K_+(z)^{-1}) \in \CFm$. 
\begin{lem}\label{lem: iso of one-dim mero eigenvalue}
$\BM[a_1^+(z),a_1^-(z)] \cong \BM[a_2^+(z),a_2^-(z)]$ in $\CF$ if and only if there exists $c \in \BC^{\times}$ such that $a_2^{\pm}(z) = c^{\pm 1}a_1^{\pm}(z)$.
\end{lem}
\subsection{Algebraic notions.} \label{ss-h algebras}
We briefly recall the notion of $\mathfrak{h}$-algebras\footnote{We use $\mathfrak{h}$ to avoid confusion of $\BC$-algebras in the usual sense.} from \cite{EV}, with $\mathfrak{h}$ being the one-dimensional complex Lie algebra and $\gamma = -\hbar$.

Let $\CV$ denote the category whose objects are $\BC$-graded $\BM$-vector spaces $X = \oplus_{\alpha\in \BC}X[\alpha]$, and morphisms are $\BM$-linear maps which preserve the $\BC$-gradings. If $X[\alpha] \neq 0$, then $\alpha$ is called a {\it weight} of $X$, non-zero vectors in $X[\alpha]$ are of weight $\alpha$, and $X[\alpha]$ is the weight space of weight $\alpha$. Let $\wt(X)$ be the set of weights of $X$.

For $X,Y \in \CV$ define their {\it dynamical tensor product} $X\wtimes Y \in \CV$ as follows. For $\alpha,\beta \in \BC$, let $X[\alpha] \wtimes Y[\beta]$ be the usual tensor product of $\BC$-vector spaces $X[\alpha]\otimes_{\BC} Y[\beta]$ modulo the relation \footnote{This dynamical tensor product and $\Phi \wtimes \Psi$ in Lemma \ref{lem: tensor product difference operators} are slightly different from \cite[\S 3.1]{EV}.}
\begin{equation}  \label{rel: weight tensor}
g(x)v \otimes_{\BC} w = v \otimes_{\BC} g(x+\beta\hbar) w \quad \mathrm{for}\ v \in X[\alpha],\ w \in Y[\beta], \ g(x) \in \BM.
\end{equation}
Let $\wtimes$ denote the image of $\otimes_{\BC}$ under the quotient. $X[\alpha]\wtimes Y[\beta]$ becomes an $\BM$-vector space by setting $g(x) (v \wtimes w) = v \wtimes g(x) w$. For $\gamma \in \BC$, set $(X\wtimes Y)[\gamma]$ to be the direct sum of the $X[\alpha] \wtimes Y[\beta]$ over all such $\alpha,\beta \in \BC$ that $\alpha+\beta = \gamma$. 

Following \cite[\S 4.1]{EV}, an $\Hlie$-{\it algebra} is a unital associative algebra $A$ over $\BC$, endowed with $\BC$-bigrading $A = \oplus_{\alpha,\beta\in \BC} A_{\alpha,\beta}$ which respects the algebra structure, and two $\BC$-algebra embeddings $\mol,\mor: \BM \longrightarrow A_{0,0}$, called the {\it left and right moment maps}, such that for $a \in A_{\alpha,\beta}$ and $g(x) \in \BM$:
$$ \mol(g(x)) a = a \mol(g(x-\alpha\hbar)),\quad \mor(g(x)) a = a \mor(g(x-\beta\hbar)). $$
 A morphism of $\Hlie$-algebras is a $\BC$-algebra homomorphism preserving the moment maps. From two $\Hlie$-algebras $A,B$ we construct their tensor product $A\dt B$ as follows. For $\alpha,\beta,\gamma \in \BC$, let $A_{\alpha,\beta}\dt B_{\beta,\gamma}$ be $A_{\alpha,\beta} \otimes_{\BC} B_{\beta,\gamma}$ modulo the relation 
$$\mor^A(g(x)) a \otimes_{\BC} b = a \otimes_{\BC} \mol^B(g(x)) b \quad \mathrm{for}\ a \in A_{\alpha,\beta},\  b \in B_{\beta,\gamma},\  g(x) \in \BM.$$
$(A\dt B)_{\alpha,\gamma}$ is the direct sum of the $A_{\alpha,\beta} \dt B_{\beta,\gamma}$ over $\beta \in \BC$. Multiplication in $A\dt B$ is induced by $(a\dt b)(a'\dt b') = aa' \dt bb'$. The moment maps are given by ($\dt$ denotes the image of $\otimes_{\BC}$ under the quotient $\otimes_{\BC} \longrightarrow \dt$)
$$\mol^{A\dt B}: g(x) \mapsto \mol^A(g(x)) \dt 1, \quad \mor^{A\dt B}: g(x) \mapsto 1 \dt \mor^B(g(x))\quad \mathrm{for}\ g(x) \in \BM. $$

To $X \in \CV$ is attached an $\Hlie$-algebra $D_X$, a $\BC$-subalgebra of $\End_{\BC}(X)$ as in \cite[\S 4.2]{EV}. For $\alpha,\beta \in \BC$, the subspace $(D_X)_{\alpha,\beta}$ consists of such $\Phi \in \End_{\BC}(X)$ that: 
$$ \Phi(X[\gamma]) \subseteq X[\gamma+\beta-\alpha],\quad  \Phi(g(x) v) = g(x+\beta\hbar) \Phi(v) $$
for $\gamma \in \wt(X),\ v \in X$ and $g(x) \in \BM$.
The moment maps $\mor,\mol$ are defined by: 
$$\mor(g(x)) v = g(x)v,\quad \mol(g(x)) v = g(x+\hbar\alpha) v \quad \mathrm{for}\ v \in X[\alpha],\ g(x) \in \BM.$$
Let $X,Y \in \CV$. For $\Phi \in (D_X)_{\alpha,\beta}$ and $\Psi \in (D_Y)_{\beta,\gamma}$, the $\BC$-linear map
\begin{equation*} 
\Pi: X \otimes_{\BC} Y \longrightarrow X \wtimes Y, \quad v \otimes_{\BC} w \mapsto \Phi(v) \wtimes \Psi(w)
\end{equation*}
respects Relation \eqref{rel: weight tensor}. Indeed, for $w \in Y[\delta]$ and $g(x) \in \BM$:
\begin{align*}
\Pi(g(x) v\otimes_{\BC} w) &= \Phi (g(x)v) \wtimes \Psi(w) = g(x+\beta\hbar) \Phi(v) \wtimes \Psi(w) \\
                                   &= \Phi(v) \wtimes g(x + \beta\hbar + (\delta+\gamma-\beta)\hbar) \Psi(w) 
                                  = \Pi(v \otimes_{\BC} g(x+\delta\hbar) w).
\end{align*}
$\Pi$ induces the $\BC$-linear map $\Phi \wtimes \Psi: X \wtimes Y \longrightarrow X \wtimes Y$ which is easily shown to be in $(D_{X\wtimes Y})_{\alpha,\gamma}$. As in \cite[Lemma 4.3]{EV}:
\begin{lem} \label{lem: tensor product difference operators}
$\Phi \dt \Psi \mapsto \Phi \wtimes \Psi$ extends uniquely to a morphism of $\Hlie$-algebras $\theta_{XY}: D_X \dt D_Y \longrightarrow D_{X\wtimes Y}$.
\end{lem}

\subsection{Elliptic quantum group.}  \label{ss-elliptic}
Fix a complex number $\tau \in \BC$ with $\mathrm{Im}(\tau) > 0$. Define the Jacobi theta function 
$$ \theta(z) := - \sum_{j=-\infty}^{\infty} \exp\left(\BI \pi (j+\frac{1}{2})^2\tau + 2\BI\pi (j+\frac{1}{2})(z+\frac{1}{2})\right), \quad \BI = \sqrt{-1}.  $$
It is an entire function on $z \in \BC$ with zeros lying on the lattice $\BZ + \BZ \tau$ and 
$$\theta(z+1) = -\theta(z),\quad \theta(z+\tau) = -e^{-\BI\pi \tau-2\BI \pi z} \theta(z), \quad \theta(-z) = -\theta(z).  $$
{\it Assume from now on that $\BZ + \BZ\tau$ and $\hbar\BZ$ only intersect at $0$. Unless otherwise stated (in the appendix), $\otimes$ denotes the ordinary tensor product $\otimes_{\BM}$ of $\BM$-vector spaces and $\BM$-linear maps.  }

Let $\BV \in \CV$ be such that $\BV = \BM v_+ \oplus \BM v_-$ with $v_{\pm}$ being of weight $\pm 1$ respectively. Define the $\End_{\BM}(\BV^{\otimes 2})$-valued meromorphic function on $z \in \BC$ by its matrix with respect to the $\BM$-basis $(v_+\otimes v_+, v_+\otimes v_-, v_-\otimes v_+, v_-\otimes v_-)$:
\begin{equation}  \label{equ: elliptic R matrix gl(2)}
\BR(z;x) = \begin{pmatrix}
1 & 0 & 0 & 0 \\
0 & \frac{\theta(z)\theta(x+\hbar)\theta(x-\hbar)}{\theta(z+\hbar)\theta(x)^2} & \frac{\theta(z+x)\theta(\hbar)}{\theta(z+\hbar)\theta(x)} & 0 \\
0 & -\frac{\theta(z-x)\theta(\hbar)}{\theta(z+\hbar)\theta(x)} & \frac{\theta(z)}{\theta(z+\hbar)} & 0 \\
0 & 0 & 0 & 1
\end{pmatrix}.
\end{equation}
It is the matrix $R^+(z,x)$ in \cite[(93)]{EF}. Applying a suitable gauge transformation in \cite[Lemma 8.1]{EF}, one obtains $R(x,z,-\frac{\hbar}{2},\tau)$ in \cite[\S 1]{FV1}.

$\BR(z;x)$ satisfies the {\it quantum dynamical Yang--Baxter equation} in \cite{EV}:
\begin{multline}  \label{equ: YBE elliptic R gl(2)}
\quad \quad \BR^{12}(z-w;x +\hbar h^{(3)}) \BR^{13}(z;x) \BR^{23}(w;x+\hbar h^{(1)}) \\
= \BR^{23}(w;x)\BR^{13}(z;x+\hbar h^{(2)})\BR^{12}(z-w;x) \in \End_{\BM}(\BV^{\otimes 3}).
\end{multline}
Here $\BR^{12}(z;x+\hbar h^{(3)})$ means that if $w_1,w_2 \in \BV$ and $w_3 \in \BV[\alpha]$, then 
$$ \BR^{12}(z;x+\hbar h^{(3)}) (w_1\otimes w_2\otimes w_3) = \BR(z;x+\hbar\alpha)\left(w_1\otimes w_2\right) \otimes w_3. $$
The other symbols have a similar meaning.

The {\it elliptic quantum group} $\CE = \CE_{\tau,\hbar}(\mathfrak{gl}_2)$ is the operator algebra in \cite[\S 3]{FV1}, or equivalently the dynamical quantum group associated to $\BR(z;x)$ in \cite[\S 4.4]{EV}. It is an $\Hlie$-algebra generated by the $L_{ij}(z) \in \CE_{i1,j1}$ with $i,j \in \{ \pm\}$ subject to the relation
\begin{equation}  \label{rel: elliptic RLL}
\mol\left(\BR^{12}(z-w;x)\right) L^{13}(z) L^{23}(w) = L^{23}(w)L^{13}(z)\mor\left(\BR^{12}(z-w;x)\right). 
\end{equation}
By \cite[\S 1]{FV1} and \cite[Proposition 4.2]{EV}, there is an $\Hlie$-algebra morphism 
\begin{equation} \label{rel: coproduct}
\Delta: \CE \longrightarrow \CE \dt \CE,\quad L_{ij}(z) \mapsto \sum_{k=\pm} L_{ik}(z) \dt L_{kj}(z)
\end{equation}
which is co-associative $(1\dt \Delta) \Delta = (\Delta \dt 1)\Delta$. For $u \in \BC$, 
\begin{equation} \label{rel: spectral automorphism}
\Psi_u: \CE \longrightarrow \CE, \quad L_{ij}(z) \mapsto L_{ij}(z+u\hbar)
\end{equation}
extends uniquely to an automorphism of $\Hlie$-algebra.\footnote{Strictly speaking $\CE$ is not well-defined as an $\mathfrak{h}$-algebra. Nevertheless, we are only concerned with representations and Equations \eqref{rel: elliptic RLL}--\eqref{rel: spectral automorphism} make sense. Equations \eqref{rel: elliptic RLL}, \eqref{rel: RLL explicit} and \cite[(4.4.3)]{EV} are coherent as $R_{pq}^{ij}(z;x) = R_{pq}^{ij}(z;x+p\hbar+q\hbar)$ for $i,j,p,q \in \{\pm\}$. }

A {\it representation} of $\CE$ on $X \in \CV$ is an $\Hlie$-algebra morphism $\rho: \CE \longrightarrow D_X$ which depends meromorphically on $z \in \BC$. More precisely, it consists of four operators $L_{ij}^X(z) \in (D_X)_{i1,j1}$ for $i,j \in \{ \pm\}$ with parameter $z \in \BC$ such that: with respect to an $\BM$-basis of $X$, for any $v \in X$, the coefficients of the $L_{ij}^X(z) v$ are meromorphic functions on $z,x$; Equation \eqref{rel: elliptic RLL} holds in $D_X$ with $\mol,\mor$ being moment maps in $D_X$. We also call $X$ an $\CE$-module.

If $(\rho,X),(\sigma,Y)$ are representations of $\CE$, then Equation \eqref{rel: coproduct} together with Lemma \ref{lem: tensor product difference operators} endows $X\wtimes Y$ with a representation $\theta_{XY} \circ (\rho \dt \sigma) \circ \Delta$.

Let $(\rho,X)$ be a representation. Fix a weight basis $(w_{\alpha})$ of $X$. For $i,j \in \{ \pm\}$, define $L_{ij}^X(z;x) \in \End_{\BM}(X)$ by $L_{ij}^X(z;x) w_{\alpha} = L_{ij}^X(z) w_{\alpha}$ for all $\alpha$. Set 
$$L^X(z;x) := \sum_{i,j \in \{ \pm\}} E_{ij} \otimes L_{ij}^X(z;x) = \begin{pmatrix}
L_{++}^X(z;x) & L_{+-}^X(z;x) \\
L_{-+}^X(z;x) & L_{--}^X(z;x)
\end{pmatrix} \in \End_{\BM}(\BV \otimes X)$$
where $E_{ij} \in \End_{\BM}(\BV)$ is $v_k \mapsto \delta_{jk}v_i$. That $\rho$ is an $\Hlie$-algebra morphism is equivalent to the following RLL relation \cite[Proposition 4.5]{EV}:
\begin{multline} \label{rel: RLL dynamical}
\quad \quad \BR^{12}(z-w;x +\hbar h^{(3)}) L^{X,13}(z;x) L^{X,23}(w;x+\hbar h^{(1)})\\
 = L^{X,23}(w;x)\,L^{X,13}(z;x+\hbar h^{(2)})\,\BR^{12}(z-w;x) \in \End_{\BM}(\BV\otimes \BV \otimes X).
\end{multline} 
For $i,j,m,n \in \{ \pm\}$, let $R_{mn}^{ij}(z;x)$ be the coefficient of $v_m \otimes v_n$ in $\BR(z;x)(v_i\otimes v_j)$. Then Equation \eqref{rel: RLL dynamical} means the following identities for all $i,j,m,n \in \{ \pm\}$, 
\begin{multline} \label{rel: RLL explicit}
\quad \quad \sum_{p,q}  R^{pq}_{mn}(z-w;x+\hbar h) L_{pi}^X(z;x) L_{qj}^X(w;x+i\hbar) \\ 
= \sum_{p,q} L_{nq}^X(w;x)L_{mp}^X(z;x+q\hbar)  R_{pq}^{ij}(z-w;x) \in \End_{\BM}(X).
\end{multline}

By Equation \eqref{equ: YBE elliptic R gl(2)}, $L^{\BV}(z;x) := \BR(z;x)$ affords a representation of $\CE$ on $\BV$.

\subsection{Asymptotic representations.} Let $\ell \in \BC$ and $\CW^{\ell} = \oplus_{j=0}^{\infty} \BM w_j$ with $w_j$ being of weight $\ell-2j$ so that $\CW^{\ell} \in \CV$. Let $L^{\ell}(z;x) \in \End_{\BM} (\BV \otimes \CW^{\ell})$ be the matrix 
\begin{equation*}
\begin{pmatrix}
\sum\limits_{j=0}^{\infty} E_{jj} \frac{\theta(z+(\ell-j+1)\hbar)\theta(x+(\ell-j+1)\hbar)\theta(x-j\hbar)}{\theta(x)\theta(x+(\ell-2j+1)\hbar)} & \sum\limits_{j=0}^{\infty} E_{j+1,j}\frac{\theta(z+x+(\ell-j)\hbar)\theta((\ell-j)\hbar)}{\theta(x+(\ell-2j-1)\hbar)} \\
-\sum\limits_{j=1}^{\infty} E_{j-1,j}\frac{\theta(z-x+j\hbar)\theta(j\hbar)}{\theta(x)} & \sum\limits_{j=0}^{\infty} E_{jj} \theta(z+(j+1)\hbar) 
\end{pmatrix}.
\end{equation*}
Here $E_{ij} \in \End_{\BM}( \CW^{\ell})$ is $w_k \mapsto \delta_{jk}w_i$. 
\begin{prop} \label{prop: finit-dim representations}
$L^{\ell}(z;x)$ and the basis $(w_j)$ define a representation of $\CE$ on $\CW^{\ell}$.
\end{prop}
\begin{proof}
Assume $n \in \BZ_{>0}$. Let $V^{n} = \oplus_{j=0}^{n} \BM v_j$ with $v_j$ being of weight $n-2j$. Under the correspondence $[u] \mapsto \theta(\hbar u), r \mapsto \hbar^{-1}, u \mapsto \hbar^{-1} z, P=s \mapsto \hbar^{-1} x$, the matrix $R^{+}(u,s)$ in \cite[(2.18)]{K1} is identified with $\BR(z;x)$. From \cite[(2.19)--(2.20) \& Theorem 4.14]{K1} (setting $\varphi_n(u-v) = -1$ in $\widehat{\pi}_{n,q^{-n-1}}$) one obtains a representation $\rho^n$ of $\CE$ on $V^n$ whose matrix with respect to the basis $(v_j)_{0\leq j \leq n}$ is 
$$\begin{pmatrix}
\frac{\theta(z+\frac{h+n+2}{2}\hbar) \theta(x+\frac{h-n}{2}\hbar)\theta(x+\frac{h+n+2}{2}\hbar)}{\theta(x)\theta(x+(h+1)\hbar)} & S^- \frac{\theta(z+x+\frac{h+n}{2}\hbar)\theta(\frac{n-h+2}{2}\hbar)}{\theta(x+(h-1)\hbar)} \\
-S^+ \frac{\theta(z-x+\frac{n-h}{2}\hbar)\theta(\frac{n+h+2}{2}\hbar)}{\theta(x)} & \theta(z+\frac{n-h+2}{2}\hbar)
\end{pmatrix} \in \End_{\BM}(\BV \otimes V^n).$$
Here $h (v_j) = (n-2j)v_j, S^{\pm}(v_j) = v_{j\mp 1}$ and $v_{-1}=v_{n+1} = 0$. Set 
$$v_0' := v_0,\quad v_j' := \frac{\theta(\hbar)\theta(2\hbar) \cdots \theta(j\hbar)}{\theta(n\hbar)\theta((n-1)\hbar) \cdots \theta((n-j+1)\hbar)} v_j \quad \mathrm{for}\ 1\leq j \leq n. $$
Then $(v_j')$ forms another basis of $V^n$. Let $L'^{n}(z;x)$ be the matrix of $\rho^n$ with respect to $(v_j')$. Let us view $V^n$ as a subspace of $\CW^n$ by $v_j' = w_j$ for $0\leq j \leq n$. 
\begin{itemize}
\item[(I)] $L_{pq}'^{n}(z;x)|_{V^n} = L_{pq}^n(z;x)|_{V^n}$ for $p,q \in \{\pm\}$ with $(p,q) \neq (+,-)$.
\item[(II)] $L_{+-}'^{n}(z;x) w_j = L_{+-}^n(z;x)w_j$ for $0\leq j < n$.
\end{itemize}

Our goal is prove Equation \eqref{rel: RLL explicit} with $L^X(z;x) = L^{\ell}(z;x)$ and $X = \CW^{\ell}$. Fix $i,j,m,n\in \{ \pm\}$ and $k \in \BZ_{\geq 0}$. Let $L$ (resp. $R$) be the left-hand side (resp. right-hand side) of Equation \eqref{rel: RLL explicit} applied to $w_k$. We are reduced to prove $L = R$. Notice that $L,R$ are linear combinations of $w_k,w_{k\pm 1}, w_{k\pm 2} \in V^{k+2} \subset \CW^{\ell}$ whose coefficients are meromorphic functions on $z,w,x,\ell \in \BC$. (Set $w_{-1} = w_{-2} = 0$.) We shall fix $z,w,x \in \BC$ to be generic and view $L = L(\ell)$ and $R = R(\ell)$ as vector-valued meromorphic functions on $\ell$. By (I)--(II), $L(n) = R(n)$ for $n > k + 2$.

For two vector-valued meromorphic functions $f(\ell),g(\ell)$ on $\ell$, let us write $f(\ell) \sim g(\ell)$ if there exist $N \in \BZ_{\geq 0}$ and $a \in \BC$ such that for $H = f,g$:
$$ H(\ell + \hbar^{-1}) = (-1)^N H(\ell),\quad H(\ell + \hbar^{-1}\tau) = (-1)^N e^{-\BI\pi(N\tau+2N\ell \hbar + a)} H(\ell). $$
Assume $L(\ell) \sim R(\ell)$. The coefficients of $L(\ell)-R(\ell)$ are products of the $\theta(\ell \hbar + b)^{\pm 1}$ for $b \in \BC$ with constant functions. Since $\theta(\ell\hbar + b)$ cannot have zeroes at large enough integers, $L(\ell) = R(\ell)$. We are led to prove $L(\ell) \sim R(\ell)$.

If $m=n=-$, then $R,L$ are independent of $\ell$ as is so the second row of the square matrix $L^{\ell}(z;x)$. If $m = n = +$, then $L(\ell),R(\ell) \sim \theta(\ell \hbar + z + \hbar) \theta(\ell \hbar + w + \hbar)$, because the first row of $L^{\ell}(z;x)\sim\theta(\ell \hbar + z + \hbar)$.

For $(m,n)=(+,-)$, we have $L(\ell) = R_{+-}^{+-}L_{+i}L_{-j} + R_{+-}^{-+} L_{-i}L_{+j}$ and
\begin{align*}
R_{+-}^{+-}L_{+i}L_{-j} &\sim \frac{\theta(\ell\hbar+x+(1-2i-2j)\hbar)\theta(\ell\hbar+x-(1+2i+2j)\hbar)}{\theta(\ell\hbar + x - (2i+2j)\hbar)^2} \theta(\ell\hbar+z+\hbar)   \\
& \sim \theta(\ell \hbar + z + \hbar), \\
R_{+-}^{-+} L_{-i}L_{+j} &\sim \frac{\theta(\ell \hbar+z-w+x-(2i+2j)\hbar)}{\theta(\ell\hbar+x-(2i+2j)\hbar)} \theta(\ell\hbar+w+\hbar) \sim \theta(\ell\hbar + z + \hbar).
\end{align*} 
Here for simplicity $i \in \{\pm\}$ is taken as $i1 \in \{\pm 1\}$ and the factors irrelevant to $\ell$ have been omitted. $R(\ell)$ is a sum of the $L_{-q}L_{+p}R_{pq}^{ij}$. Since $L_{-q}, R_{pq}^{ij} \sim 1$ and $L_{+p} \sim \theta(\ell \hbar + z + \hbar)$, we have $R(\ell) \sim \theta(\ell \hbar + z + \hbar) \sim L(\ell)$.

For $(m,n) = (-+)$, similarly $L(\ell) \sim \theta(\ell\hbar + w + \hbar) \sim R(\ell)$.
\end{proof}
One could have deduced the $\CW^{\ell}$ from the evaluation modules in \cite[\S 4]{FV1} by suitable gauge transformations. We took an indirect approach by analyzing the $\ell$-dependence of matrix coefficients. This is to be compared with \cite[Proposition 4.5]{HJ} and \cite[Lemma 5.1]{Z2} where matrix coefficients are Laurent polynomials in $\ell := q^k$ in the case of quantum affine (super)algebras.
 \begin{rem}  \label{rem: simplicity asymptotic}
 If $\ell \notin \BZ_{\geq 0} + \hbar^{-1}(\BZ + \BZ\tau)$, then $\CW^{\ell}$ is simple.
 If $\ell \in l + \hbar^{-1}(\BZ + \BZ\tau)$ with $l \in \BZ_{\geq 0}$, then $V^{\ell} := \oplus_{j=0}^l\BM w_j$ is a sub-$\CE$-module of $\CW^{\ell}$ and it is contained in any non-zero sub-$\CE$-module of $\CW^{\ell}$. In other words, $V^{\ell}$ is the simple socle of $\CW^{\ell}$.
 \end{rem}
 \begin{defi}
 Let $\ell, u \in \BC$. The {\it asymptotic representation} $\CW^{\ell,u}$ is the pullback of the representation $\CW^{\ell}$ in Proposition \ref{prop: finit-dim representations} by $\Psi_u$ in Equation \eqref{rel: spectral automorphism}. $\ell$ is called the {\it spin parameter} and $u$ the {\it spectral parameter}.
 \end{defi}
 
\section{Category $\BGG$ and $q$-characters}\label{s-BGG}
We introduce a tensor category $\BGG$ of representations of $\CE$ containing all the $\CW^{\ell,u}$, and study its Grothendieck ring, which turns out to be commutative. 

Let $X$ be an $\CE$-module. For $\alpha \in \wt(X)$ choose an $\BM$-basis $(v_j^{\alpha})_{1\leq j \leq r}$ of $X[\alpha]$. For $1\leq i,j \leq r$ let $a_{ij}^{\alpha}(z;x)$ be the coefficient of $v_i^{\alpha}$ in $L_{--}^X(z) v_j^{\alpha}$; these are meromorphic functions on $z,x \in \BC$ by Section \ref{ss-elliptic}. We say that $L_{--}^X(z)$ is invertible for generic $z \in \BC$ if $\det(a_{ij}^{\alpha}(z;x))_{1\leq i,j \leq r}$ is a non-zero meromorphic function for all $\alpha \in \wt(X)$.  As a two-by-two matrix over $D_X$, $L^X(z)$ admits a {\it Gauss decomposition}:
\begin{align*}  
L^X(z) := \begin{pmatrix}
L_{++}^X(z) & L_{+-}^X(z) \\
L_{-+}^X(z) & L_{--}^X(z)
\end{pmatrix} =  \begin{pmatrix}
1 & F^X(z) \\
0 & 1
\end{pmatrix} \begin{pmatrix}
K_+^X(z) & 0 \\
0 & K_-^X(z) 
\end{pmatrix} \begin{pmatrix}
1 & 0 \\
E^X(z) & 1
\end{pmatrix}.
\end{align*}
By definition, we have $K_-^X(z) = L_{--}^X(z) \in (D_X)_{-1,-1}$ and 
$$ K_+^X(z) \in (D_X)_{1,1},  \quad E^X(z) \in (D_{X})_{0,2},\quad F^X(z) \in (D_X)_{2,0}.  $$
It follows that $(X,K_+^X(z),K_-^X(z))|_{X[\alpha]} \in \CF$ for $\alpha \in \wt(X)$. 

\begin{defi} \label{def: category BGG}
A representation $X$ of $\CE$ is said to be in category $\BGG$ if:
\begin{itemize}
\item[(1)] the weight spaces $X[\alpha]$ are finite-dimensional over $\BM$;
\item[(2)] there exist $\alpha_1,\alpha_2,\cdots,\alpha_r \in \BC$ such that $\wt(X)\subseteq \cup_{j=1}^r(\alpha_j + 2\BZ_{\leq 0})$;
\item[(3)] $L_{--}^X(z)$ is invertible for generic $z \in \BC$;
\item[(4)] for $\alpha \in \wt(X)$ we have $(X,K_+^X(z),K_-^X(z))|_{X[\alpha]} \in \CFm$.
\end{itemize}
A morphism of two representations $X,Y$ in $\BGG$ is an $\BM$-linear map $\Phi: X \longrightarrow Y$ such that $\Phi L_{ij}^X(z) = L_{ij}^Y(z) \Phi$ for all $i,j \in \{ \pm\}$. 
\end{defi}
When there is no confusion, we drop the superscript $X$ from $K_{\pm}^X, E^X,F^X, L_{ij}^X$.
\begin{lem}
$\BGG$ is an abelian category.
\end{lem}
\begin{proof}
This follows from Lemma \ref{lem: abelian category meromorphic eigenvalue}.
\end{proof}
\begin{prop}  \label{prop: asymptotic representation Gauss decomposition}
The representation $\CW^{\ell}$ in Proposition \ref{prop: finit-dim representations} is in category $\BGG$.
\end{prop}
\begin{proof}
Since $L_{--}(z) w_j = \theta(z+(j+1)\hbar) w_j$ for all $j \in \BZ_{\geq 0}$, $L_{--}(z)$ is invertible. Conditions (1)--(3) in Definition \ref{def: category BGG} are true. $E,F,K_-$ are easy to compute:
\begin{gather*}
K_-(z) w_j = \theta(z+(j+1)\hbar) w_j,\quad  E(z) w_j = -\frac{\theta(z-x+(j-1)\hbar)\theta(j\hbar)}{\theta(z+j\hbar)\theta(x+\hbar)} w_{j-1}, \\
F(z) w_j = \frac{\theta(z+x+(\ell-j)\hbar)\theta((\ell-j)\hbar)}{\theta(z+(j+1)\hbar)\theta(x+(\ell-2j-1)\hbar)} w_{j+1}.
\end{gather*}
For $K_+$, let us be in the situation of the proof of Proposition \ref{prop: finit-dim representations}. From \cite[Theorem 4.13]{K1} one observes that: for $l \in \BZ_{\geq 0}$, 
$$ K_+^{V^l}(z)K_-^{V^l}(z-\hbar) v_j' = \theta(z+(l+1)\hbar) \theta(z) v_j' \quad \mathrm{for}\ 0 \leq j \leq l. $$ 
It is therefore enough to prove the above identity for $\CW^{\ell}$. In other words,
$$ (L_{++}(z)-L_{+-}(z)L_{--}(z)^{-1}L_{-+}(z)) L_{--}(z-\hbar) w_j = \theta(z+(\ell+1)\hbar) \theta(z) w_j. $$
Let $z,x \in \BC$ be generic. The left hand side is $g_j(\ell) w_j$ where $g_j(\ell)$ is an entire function on $\ell \in \BC$ and $g_j(\ell) \sim \theta(\ell \hbar + z + \hbar)$. From the embedding $V^n \subset \CW^n$ we deduce that $g_j(n) = \theta(n\hbar+z+\hbar) \theta(z)$ for all $n \in \BZ_{>j+3}$. This forces $g_j(\ell) = \theta(z+(\ell+1)\hbar)\theta(z)$ for all $\ell \in \BC$.
\end{proof}
\begin{rem}
Let $X \in \BGG$. Then $K_+^X(z)K_-^X(z-\hbar) \in (D_X)_{0,0}$ commutes with the $L_{ij}^X(w) \in D_X$; see \cite[Theorem 13]{FV1}, \cite[Remark 10]{EF} and \cite[Corollary E.24]{K2}.
\end{rem}
Next we adapt the $q$-character theory of Knight \cite{Kn} and Frenkel--Reshetikhin \cite{FR} to the category $\BGG$.
Motivated by Lemma \ref{lem: iso of one-dim mero eigenvalue}, let $\CM$ be the quotient of the set of pairs $(a^+(z),a^-(z))$ of non-zero meromorphic functions on $z \in \BC$ by the relation $$(a^+(z),a^-(z)) \equiv (c a^+(z),c^{-1}a^-(z)) \quad \mathrm{for}\ c \in \BC^{\times}.$$
 The isomorphism class of $(a^+(z),a^-(z))$ is denoted by $[a^+(z),a^-(z)] \in \CM$. Make $\CM$ into a group by component-wise multiplication. Introduce formal symbols $t^{\alpha}$ for $\alpha \in \BC$. Define the set $\CM_t$ whose elements are formal sums (possibly infinite) $\sum_{\alpha\in \BC}\sum_{m\in \CM} c_{m,\alpha} mt^{\alpha}$ with coefficients $c_{m,\alpha} \in \BZ$ such that:
\begin{itemize}
\item[(M1)]  there exist $\alpha_1,\alpha_2,\cdots,\alpha_r \in \BC$ such that the coefficient of $mt^{\alpha}$ is non-zero only if $\alpha\in \cup_{j=1}^r(\alpha_j + 2\BZ_{\leq 0})$;
\item[(M2)] for all $\alpha \in \BC$, the number of terms $mt^{\alpha}$ with non-zero coefficients is finite.
\end{itemize}
Make $\CM_t$ into a ring: addition is the usual one of formal sums; multiplication is induced by $(mt^{\alpha})(m't^{\beta}) = (mm')t^{\alpha+\beta}$ for $m,m' \in \CM$ and $\alpha,\beta \in \BC$.
\begin{defi} \label{def: q-character}
Let $X$ be in category $\BGG$. For $\alpha \in \wt(X)$ choose an $\BM$-basis $(v_{i}^{\alpha})_{1\leq i \leq r_{\alpha}}$ of $X[\alpha]$ such that: the coefficients $a_{ij}^{\alpha\pm}(z;x)$ of $v_i^{\alpha}$ in $K_{\pm}^V(z) v_j^{\alpha}$ form upper triangular matrices $(a_{ij}^{\alpha\pm}(z;x))_{1\leq i,j \leq r_{\alpha}}$ with diagonals $a_{ii}^{\alpha\pm}(z;x) = a_{ii}^{\alpha\pm}(z) \neq 0$ being independent of $x$. The $q$-character of $X$ is defined to be
$$ \qc(X) := \sum_{\alpha \in \wt(X)} \sum_{i=1}^{r_{\alpha}} [a_{ii}^{\alpha+}(z),a_{ii}^{\alpha-}(z)]t^{\alpha} \in \CM_t. $$
\end{defi}
\begin{lem}
$\qc(X)$ is  independent of the choice of basis $(v_i^{\alpha})$ of $X$. 
\end{lem}
\begin{proof}
By Lemma \ref{lem: abelian category meromorphic eigenvalue} and Condition (4) in Definition \ref{def: category BGG}, such a basis $(v_i^{\alpha})$ exists. Conditions (1)--(2) in Definition \ref{def: category BGG} implies (M1)--(M2) for $\qc(X)$, which is an element in $\CM_t$. From the upper triangular property we deduce that: in the Grothendieck group $K_0(\CFm)$ of the abelian category $\CFm$ the isomorphism class of $(X, K_+^V(z),K_-^V(z))|_{X[\alpha]}$ is the sum of those of irreducible objects $\BM[a_{ii}^{\alpha+}(z),a_{ii}^{\alpha-}(z)]$, which corresponds to the second summation in $\qc(X)$ and is therefore independent of the choice of basis $(v_i^{\alpha})$ by Lemma \ref{lem: iso of one-dim mero eigenvalue}.
\end{proof}
\begin{example}  \label{example: q-character asymptotic representations}
Let $\ell \in \BC$. From the proof of Proposition \ref{prop: asymptotic representation Gauss decomposition} we see that
\begin{align*}
\qc(\CW^{\ell}) &= \sum_{j=0}^{\infty} \left[\frac{\theta(z+(\ell+1)\hbar) \theta(z)}{\theta(z+j\hbar)}, \theta(z+(j+1)\hbar)\right] t^{\ell - 2j}  \\
&= [\theta(z+(\ell+1)\hbar), \theta(z+\hbar)]t^{\ell} \times \sum_{j=0}^{\infty} \left[\frac{\theta(z)}{\theta(z+j\hbar)}, \frac{\theta(z+(j+1)\hbar)}{\theta(z+\hbar)}\right] t^{-2j}.
\end{align*} 
In particular, $\qc(\CW^{\ell}) = [\theta(z+(\ell+1)\hbar), \theta(z+\hbar)]t^{\ell} \times \qc(\CW^0)$.
\end{example}
Let $X$ be in category $\BGG$. A non-zero weight vector $v$ is called a {\it highest weight vector}\footnote{This is stronger than the notion of highest weight in \cite[\S 3]{FV1}.} if $L_{-+}(z) v = 0$ and $L_{\pm\pm}(z)v = a^{\pm}(z) v$ where $a^{\pm}(z)$ are meromorphic functions on $z$ (independent of $x$). Call $V$ a {\it highest weight module} if it is $\BM$-linearly spanned by the $L_{+-}(z_1)L_{+-}(z_2) \cdots L_{+-}(z_n) v$ where $n \in \BZ_{\geq 0}$, $z_1,z_2,\cdots,z_n \in \BC$ are generic, and $v \in X[\alpha]$ is a highest weight vector; the monomial $[a^+(z),a^-(z)]t^{\alpha} \in \CM_t$ is called the highest weight of $X$ (and of $v$).

For example, $w_0 \in \CW^{\ell}$ is a highest weight vector. $\CW^{\ell}$ is of highest weight if and only if $\ell \notin \BZ_{\geq 0} + \hbar^{-1}(\BZ + \BZ \tau)$. See Remark \ref{rem: simplicity asymptotic}.
\begin{lem}  \label{lem: highest weight}
Simple objects in category $\BGG$ are of highest weight. Two such objects are isomorphic if and only if their highest weights are the same, if and only if their $q$-characters are the same.
\end{lem}
\begin{proof}
Let $S$ be a simple object in category $\BGG$. By Condition (2) of Definition \ref{def: category BGG}, there exists $\alpha \in \wt(S)$ such that  $S[\alpha+n] = 0$ for all $n \in \BZ_{>0}$. This implies $L_{-+}(z) S[\alpha] = 0$ and $K_{\pm}(z)|_{S[\alpha]} = L_{\pm\pm}(z)|_{S[\alpha]}$. Since $(S[\alpha],K_+,K_-)\in \CFm$, there exists $0 \neq v \in S[\alpha]$ such that $K_{\pm}(z) v = a^{\pm}(z) v$ with $a^{\pm}(x) \in \BM$. So $v$ is a highest weight vector. Since $S$ is simple, it is spanned by the $L_{i_1j_1}(z_1)L_{i_2j_2}(z_2) \cdots L_{i_nj_n}(z_n) v$. By Equation \eqref{rel: RLL explicit} and \cite[Lemma 2.4]{C}, one may assume $(i_s,j_s) = (+-)$ for $1\leq s \leq n$. This proves that $S$ is of highest weight. $\qc(S)$ is $[a^+(z),a^-(z)]t^{\alpha}$ plus terms of the form $m't^{\alpha-2n}$ with $m' \in \CM$ and $n \in \BZ_{>0}$. The highest weight $[a^+(z),a^-(z)]t^{\alpha}$ of $S$ appears in $\qc(S)$ as the leading term. 

Assume that $mt^{\alpha}$ is the highest weight of two simple objects $S_1,S_2$ in category $\BGG$. Consider the $\CE$-module $V = S_1\oplus S_2$ with natural projections $\pi_i: V \longrightarrow S_i$. Let $v_i \in S_i$ be highest weight vectors. Then $(v_1,v_2) \in V$ is also a highest weight vector, which generates a highest weight submodule $W$ of $V$. Since $\pi_i(v_1,v_2) = v_i$, the restrictions $\pi_i|_W: W \longrightarrow S_i$ are non-zero and hence surjective as the $S_i$ are simple. This implies that the $S_i$ are simple quotients of $W$. Being a highest weight module, $W$ has a unique simple quotient. This implies $S_1 \cong S_2$. 
\end{proof}
For $mt^{\alpha} \in \CM_t$ a highest weight of an object in category $\BGG$, fix a simple object $S(mt^{\alpha}) \in \BGG$ of highest weight $mt^{\alpha}$. Let $\Simple$ be the set of all such $S(mt^{\alpha})$. 

Define the completed Grothendieck group $K_0(\BGG)$: elements are formal sums (possibly infinite) $\sum_{S\in \Simple} n_S[S]$ with coefficients $n_S \in \BZ$ such that $\oplus_{S\in \Simple} S^{\oplus |n_S|}\in \BGG$; addition is the usual one of formal sums. For $X \in \BGG$ and $S \in \Simple$, the multiplicity $m_{S,X} \in \BZ_{\geq 0}$ of $S$ in $X$ is a well-defined due to Definition \ref{def: category BGG} (1)--(2), as in the case of Kac--Moody algebras \cite[\S 9.6]{Kac}. Furthermore $[X] := \sum_{S\in \Simple} m_{S,X}[S] \in K_0(\BGG)$. By Definition \ref{def: q-character}, $[X] \mapsto \qc(X)$ extends uniquely to a morphism of additive groups $\qc: K_0(\BGG) \longrightarrow \CM_t$.
\begin{prop} \label{prop: tensor product q-character}
If $X,Y \in \BGG$, then the tensor product representation $X\wtimes Y$ is in category $\BGG$ and $\qc(X \wtimes Y) = \qc(X)\qc(Y)$.
\end{prop}
\begin{proof}
Conditions (1)--(2) in Definition \ref{def: category BGG} are clear for the representation $X\wtimes Y$. The proof of Condition (3) for $X\wtimes Y$ is similar to \cite[\S 2.4]{FR}. For $\alpha,\beta \in \BC$, choose ordered bases $(v_i^{\alpha})_{1\leq i \leq r_{\alpha}}$ and $(w_j^{\beta})_{1\leq j \leq s_{\beta}}$ for $X[\alpha]$ and $Y[\beta]$ respectively as in Definition \ref{def: q-character}. Order the basis $(v_{i}^{\alpha} \wtimes w_{j}^{\beta})_{\alpha,\beta,i,j}$ of $X\wtimes Y$ such that:
\begin{itemize}
\item[(a)] $v_{i_1}^{\alpha}\wtimes w_{j_1}^{\beta} \preceq v_i^{\alpha} \wtimes v_j^{\beta}$ if $i_1 \leq i \leq r_{\alpha}$ and $ j_1 \leq j  \leq s_{\beta}$;
\item[(b)] $v_{i}^{\alpha+2} \wtimes w_{j}^{\beta-2} \prec v_{k}^{\alpha} \wtimes w_l^{\beta}$ if $ i \leq r_{\alpha+2},\  j \leq s_{\beta-2},\ k \leq r_{\alpha}$ and $l  \leq s_{\beta}$.
\end{itemize}
By Equation \eqref{rel: coproduct} and Lemma \ref{lem: tensor product difference operators}, in $D_{X \wtimes Y}$ we have
 $$ K_-^{X\wtimes Y}(z) =  K_-^X(z) \wtimes K_-^Y(z) + L_{-+}^X(z) \wtimes L_{+-}^Y(z).  $$
Since $L_{-+}^X(z) \in (D_X)_{-1,1}$ and $L_{+-}^Y(z) \in (D_Y)_{1,-1}$, the ordered basis $(v_i^{\alpha} \wtimes w_j^{\beta})$ induces an upper triangular matrix for $K_-^{X\wtimes Y}(z)$, whose diagonal element associated to $v_i^{\alpha} \wtimes w_j^{\beta}$ is the product of those associated to $v_i^{\alpha}, w_j^{\beta}$. So $K_-^{X\wtimes Y}(z)$ is invertible. 

Conditions (3)--(4) for $X$ indicate that the $2\times 2$ matrix $L^X(z)$ is invertible. Set
$$ L^X(z)^{-1} = \begin{pmatrix}
L_{++}^{X*}(z) & L_{+-}^{X*}(z) \\
L_{-+}^{X*}(z) & L_{--}^{X*}(z)
\end{pmatrix}. $$ 
Then $L_{ij}^{X*}(z) \in (D_X)_{i1,j1}$ for $i,j \in \{ \pm\}$ and $L_{++}^{X*}(z) = K_+^X(z)^{-1}$. By Equation \eqref{rel: coproduct}, $L^{X\wtimes Y}(z)$ is also invertible, and we have in $D_{X\wtimes Y}$:
$$ K_+^{X\wtimes Y}(z)^{-1} =  K_+^X(z)^{-1} \wtimes K_+^Y(z)^{-1} + L_{-+}^{X*}(z) \wtimes L_{+-}^{Y*}(z).   $$
So similar arguments for $K_-^{X\wtimes Y}(z)$ work for $K_+^{X\wtimes Y}(z)^{-1}$. In particular Condition (4) is true for $X\wtimes Y$ and the $q$-character formulas match.
\end{proof}
 $K_0(\BGG)$ is a ring with multiplication induced from $[X][Y] := [X\wtimes Y]$ for $X,Y \in \BGG$. 
\begin{cor} \label{cor: q-character ring homo}
$\qc: K_0(\BGG) \longrightarrow \CM_t$ is an injective homomorphism of rings. In particular, $K_0(\BGG)$ is a commutative ring.
\end{cor}
The proof of injectivity is standard as in \cite[Theorem 3 (1)]{FR}, making use of Lemma \ref{lem: highest weight}. We are able to prove the first main result of this paper.
\begin{theorem} \label{thm: identity in K asymptotic representations}
$[\CW^{\ell,0} \wtimes \CW^{0,u}] = [\CW^{\ell-u,u} \wtimes \CW^{u,0}] \in K_0(\BGG)$ for $\ell,u \in \BC$. 
\end{theorem}
\begin{proof}
Let us make explicit the variable $z$ in elements of $\CM_t$. Set $m(z) := \qc(\CW^{0,0})$. From Example \ref{example: q-character asymptotic representations} and Equation \eqref{rel: spectral automorphism} we obtain
$$\qc(\CW^{\ell,u}) = m(z+u\hbar) [\theta(z+(\ell+u+1)\hbar), \theta(z+(u+1)\hbar)] t^{\ell} $$
and $\qc(\CW^{\ell,0} \wtimes \CW^{0,u}) = \qc(\CW^{\ell-u,u} \wtimes \CW^{u,0})$. Conclude by Corollary \ref{cor: q-character ring homo}.
\end{proof}
The spin/spectral parameters in asymptotic representations are interchangeable. This observation will lead to functional relations in Section \ref{s-TQ}.
\section{Factorization of simple representations}\label{s-Simple}
We prove more general facts on simple objects in category $\BGG$. 

By Lemma \ref{lem: highest weight}, simple objects in category $\BGG$ are parametrized by their highest weights. The following result, stated without proof in \cite[Theorem 9]{FV1}, describes all such highest weights as elements in $\CM_t$. 
\begin{theorem} \label{thm: simple modules in O}
$mt^{\alpha} = [a^+(z),a^-(z)] t^{\alpha} \in \CM_t$ is the highest weight of a simple module in $\BGG$ if and only if there exist $\lambda \in \BC^{\times}$ and $\alpha_1,\alpha_2,\cdots,\alpha_n, \beta_1,\beta_2,\cdots,\beta_n \in \BC$ such that $\alpha_1+\alpha_2+ \cdots + \alpha_n - \beta_1 - \beta_2 - \cdots - \beta_n = \alpha$ and
$$ \frac{a^+(z)}{a^-(z)} = \lambda \prod_{k=1}^n \frac{\theta(z+\alpha_k\hbar)}{\theta(z+\beta_k\hbar)}. $$
\end{theorem}
\begin{proof}
Let $S$ be a simple module in $\BGG$ of highest weight $mt^{\alpha}$ with $v \in S[\alpha]$ being a highest weight vector. Choose an $\BM$-basis $(v_1,v_2,\cdots,v_n)$ of $S[\alpha-2]$. Then there exist $A_k(z;x), B_k(z;x)$ meromorphic functions on $z,x$ such that
$$ L_{+-}(z) v = \sum_{k=1}^n A_k(z;x) v_k,\quad L_{-+}(z) v_k = B_k(z;x) v $$
Applying Equation \eqref{rel: RLL explicit} with $(ij) = (-+) = (nm)$ to $v$, we obtain
\begin{align*}
& \frac{\theta(z-w+x+\alpha\hbar)\theta(\hbar)}{\theta(z-w+\hbar)\theta(x+\alpha\hbar)} a^-(z)a^+(w)- \frac{\theta(z-w+x)\theta(\hbar)}{\theta(z-w+\hbar)\theta(x)} a^-(w)a^+(z) \\
=& \sum_{k=1}^n \frac{\theta(z-w)}{\theta(z-w+\hbar)} A_k(z;x+\hbar)B_k(w;x).
\end{align*}
Set $h(z) := \frac{a^+(z)}{a^-(z)}$. Multiply both sides of the identity by $\frac{\theta(z-w+\hbar)}{a^-(z)a^-(w)\theta(\hbar)}$: 
$$ \frac{\theta(z-w+x+\alpha\hbar)}{\theta(x+\alpha\hbar)} h(w) - \frac{\theta(z-w+x)}{\theta(x)} h(z) = \frac{\theta(z-w)}{\theta(\hbar)} \sum_{k=1}^n \frac{A_k(z;x+\hbar)B_k(w;x)}{a^-(z)a^-(w)}. $$
By setting $w = z+1$ and $w=z+\tau$ respectively, we obtain
$$ h(z+1) = h(z),\quad h(z+\tau) = e^{-2 \BI\pi \alpha\hbar} h(z). $$
Such a non-zero meromorphic function must be of the form in the theorem.

The proof of the \lq\lq if\rq\rq\ part is standard, by taking tensor products of the $\CW^{\ell,u}$ and one-dimensional representations of highest weight $[g(z),g(z)]t^0$ for $g(x) \in \BM^{\times}$. 
\end{proof}
For $\alpha,\beta \in \BC$, let $L(\alpha,\beta) := S([\theta(z+\alpha \hbar),\theta(z+\beta \hbar)]t^{\alpha-\beta})$; by Remark \ref{rem: simplicity asymptotic} it is the submodule of $\CW^{\alpha-\beta,\beta-1}$ generated by $w_0$. Define
$$ \Sigma(\alpha,\beta) := \begin{cases}
\{\beta + p \ |\ 0 \leq p \leq l-1, p \in \BZ \} & \mathrm{if}\ \alpha-\beta \in l + \hbar^{-1}(\BZ+\BZ\tau), \ l \in \BZ_{\geq 0}, \\
\{\beta + p\ |\ p \in \BZ_{\geq 0} \} & \mathrm{otherwise}.
\end{cases} $$
If $p > 0$, then $w_p \in L(\alpha,\beta)$ implies $\beta+p-1\in \Sigma(\alpha,\beta)$. 

The following result and its proof are adapted from half of \cite[Proposition 3.6]{M}.
\begin{prop} \label{prop: co-cyclicity}
Let $n \in \BZ_{>0}$ and $\alpha_1,\alpha_2,\cdots,\alpha_n, \beta_1,\beta_2,\cdots,\beta_n \in \BC$ be such that:
 $\alpha_j - \Sigma(\alpha_i,\beta_i)$ and $\hbar^{-1}(\BZ + \BZ\tau)$ do not intersect for all $1\leq i < j \leq n$.
 Then the tensor product $L(\alpha_1,\beta_1) \wtimes L(\alpha_2,\beta_2) \wtimes \cdots \wtimes L(\alpha_n,\beta_n)$ contains a unique non-zero vector (up to scalar product by $\BM$) annihilated by $L_{-+}(z)$.
\end{prop}
\begin{proof}
By induction on $n$: for $n = 1$ this is trivial as $L(\alpha_1,\beta_1)$ is simple. Let $n > 1$. Suppose $0\neq v\in \wtimes_{j=1}^n L(\alpha_j,\beta_j)$ is annihilated by $L_{-+}(z)$. Set $\mathcal{L} := \wtimes_{j=2}^n L(\alpha_j,\beta_j)$ and write
$ v = \sum_{r=0}^p w_r \wtimes v_r \in \CW^{\alpha_1-\beta_1,\beta_1-1} \wtimes \mathcal{L}$
with $v_r \in \mathcal{L}$ for $0\leq r \leq p$ and $v_p \neq 0$. Computing the term $w_p \wtimes \mathcal{L}$ in $L_{-+}(z) v = 0$ gives $L_{-+}(z) v_p = 0$. The induction hypothesis applied to $\mathcal{L}$, one may assume $v_p = w_0^{\wtimes n-1}$. If $p = 0$ then $v = w_0^{\wtimes n}$ and we are done. If $p > 0$, then $\beta_1+p-1\in \Sigma(\alpha_1,\beta_1)$. Consider the component $w_{p-1} \wtimes \mathcal{L}$ in $L_{-+}(z)v = 0$: 
\begin{align*}
&  \prod_{j=2}^n \theta(z+\alpha_j\hbar) \left(\frac{\theta(z+(\beta_1-1+p)\hbar-x) \theta(p\hbar)}{\theta(x)}w_{p-1} \wtimes v_p\right) \\
= &\ \theta(z+(p-1+\beta_1)\hbar) \left(w_{p-1}\wtimes L_{-+}(z)v_{p-1}\right).
\end{align*}
By Proposition \ref{prop: finit-dim representations}, $L_{-+}(z) v_{p-1} \in \mathcal{L}$ is entire on $z$. Set $z = - (p-1+\beta_1)\hbar$. Then
$ \prod_{j=2}^n \theta((\alpha_j-(\beta_1+p-1))\hbar) = 0$, and $\alpha_j - (\beta_1+p-1) \in \hbar^{-1}(\BZ + \BZ\tau)$ for some $2\leq j \leq n$,
in contradiction with the assumption as $\beta_1+p-1 \in \Sigma(\alpha_1,\beta_1)$.  
\end{proof}
\begin{prop} \label{prop: cyclicity}
Let $n, \alpha_i,\beta_i$ be as in Proposition \ref{prop: co-cyclicity}. Assume:
$\beta_j + \Sigma(-\beta_i,-\alpha_i)$ and $\hbar^{-1}(\BZ+\BZ\tau)$ do not intersect for all $1\leq i < j \leq n$.
Then the tensor product $L(\alpha_1,\beta_1) \wtimes L(\alpha_2,\beta_2) \wtimes \cdots \wtimes L(\alpha_n,\beta_n)$ is a highest weight module.
\end{prop}
\begin{proof}
Use induction on $n$. For $n = 1$ this is obvious. Let $n > 1$. Assume that $\mathcal{L} := \wtimes_{j=2}^{n} L(\alpha_j,\beta_j)$ is of highest weight, with highest weight vector $w_0^{\wtimes n-1}$.

$w_0 \wtimes w_0^{\wtimes n-1}$ being a highest weight vector of $L(\alpha_1,\beta_1) \wtimes \mathcal{L}$ generates a sub-$\CE$-module $S$. Let us prove by induction on $p \in \BZ_{\geq 0}$ with $w_p \in L(\alpha_1,\beta_1)$ that $w_p \wtimes w_0^{\wtimes n-1} \in S$.  For $p = 0$ this is trivial. Let $p > 0$ and $w_p \in L(\alpha_1,\beta_1)$. Then $\beta_1+p-1 \in \Sigma(\alpha_1,\beta_1)$ and $\theta((\alpha_1-\beta_1-p+1)\hbar) \neq 0$. By induction hypothesis $w_{p-1} \wtimes w_0^{\wtimes n-1} \in S$. Applying $L_{+-}(z)$ to this vector gives
\begin{align*}
& \frac{\theta(z+(\alpha_1-p+1)\hbar)\theta(x+(1-p)\hbar)\theta(x+(\alpha_1-\beta_1-p+2)\hbar)}{\theta(x)\theta(x+(\alpha_1-\beta_1-2p+3))} w_{p-1} \wtimes L_{+-}(z) w_0^{\wtimes n-1} \\
& + \prod_{j=2}^{n}\theta(z+\beta_j\hbar) \left(\frac{\theta(z+x+(\alpha_1-p)\hbar) \theta((\alpha_1-\beta_1-p+1)\hbar)}{\theta(x+(\alpha_1-\beta_1-2p+1)\hbar)}  w_p \wtimes w_0^{\wtimes n-1} \right).
\end{align*}
$L_{+-}(z) w_0^{\wtimes n-1} \in \mathcal{L}$ being entire on $z$, set $z = -(\alpha_1-p+1)\hbar$ so that the first term vanishes. By assumption, the second term is non-zero. So $w_p \wtimes w_0^{\wtimes n-1} \in S$. 

This proves $L(\alpha_1,\beta_1) \wtimes w_0^{\wtimes n-1}  \subseteq S$. The vectors $v \in \mathcal{L}$ such that $L(\alpha_1,\beta_1) \otimes v \subseteq S$ form a submodule $\mathcal{L}'$ of $\mathcal{L}$. Since $\mathcal{L}$ is generated by $w_0^{\wtimes n-1} \in \mathcal{L}'$, we have that $\mathcal{L}'= \mathcal{L}$ and $\wtimes_{j=1}^n L(\alpha_j,\beta_j)$ is a highest weight module.  
\end{proof}
Our proof is similar to those of \cite[Lemma 4.10]{CP} and \cite[Theorem 4.6]{MY} for quantum affine $\mathfrak{sl}_2$. The Vandermonde determinant arguments in {\it loc. cit.} can be simplified and strengthened by Weyl modules, as indicated in the proof of \cite[Theorem 2.6(iii)]{Chari}; see a closer situation in \cite[Proposition 5.2]{Z1}. 
\begin{rem} \label{rem: dual}
One might define the {\it twisted dual} $X^{\vee}$ of $X \in \BGG$ as in \cite[\S 3]{Z2} and \cite[\S 7]{Z1}: the pullback of the usual Hopf dual by a Cartan involution which permutes highest weights and lowest weights. Presumably $L(\alpha,\beta)^{\vee} \cong L(-\beta,-\alpha)$ up to tensor product by one-dimensional modules, and Propositions \ref{prop: co-cyclicity}--\ref{prop: cyclicity} are dual to each other. The Hopf dual of $\CE$-modules was discussed in \cite[\S 11]{FV1}.
\end{rem}
The following result is parallel to \cite[Proposition 3.6]{M} and \cite[Corollary 4.7]{MY}.
\begin{cor}  \label{cor: tensor product factorization simple O}
In Theorem \ref{thm: simple modules in O}, there exist a rearrangement of the $\alpha_k,\beta_k$ and a one-dimensional $\CE$-module $D$ of weight zero such that the simple module of highest weight $mt^{\alpha}$ is isomorphic to $D\wtimes L(\alpha_1,\beta_1) \wtimes L(\alpha_2,\beta_2) \wtimes \cdots \wtimes L(\alpha_n,\beta_n)$. 
\end{cor}
\begin{proof}
We follow the arguments right after \cite[(3.19)]{M}. Notice that the integer part $\lfloor z\rfloor$ of an element $z \in \BZ + \hbar^{-1}(\BZ+\BZ\tau)$ is well-defined by the projection 
$$ \BZ + \hbar^{-1} (\BZ+\BZ\tau) \longrightarrow \BZ, \quad n +\hbar^{-1} x \mapsto n \quad \mathrm{for}\ n \in \BZ,\ x \in \BZ+\BZ\tau. $$
One can rearrange the $\alpha_k,\beta_k$ such that: for every $1\leq k \leq n$, if 
$$X_k := \{\alpha_p-\beta_q\ |\ k \leq p, q \leq n \} \cap (\BZ + \hbar^{-1} (\BZ+\BZ\tau)) \neq \emptyset, $$
then $\alpha_k - \beta_k \in X_k$ and $\lfloor \alpha_k - \beta_k \rfloor \leq \lfloor z \rfloor$ for all $z \in X_k$.
The conditions in Propositions \ref{prop: co-cyclicity}--\ref{prop: cyclicity} are fulfilled. $\wtimes_{k=1}^n L(\alpha_k,\beta_k)$ is of highest weight and contains a unique highest weight vector; it must be simple by Lemma \ref{lem: highest weight}.

Set $g(z) = a^+(z) \prod_{k=1}^n \theta(z+\alpha_k\hbar)^{-1}$ and let $D$ be the simple module of highest weight $[g(z),g(z)]$. Then $D$ is one-dimensional, and $D \wtimes (\wtimes_{k=1}^n L(\alpha_k,\beta_k))$ is a simple module of highest weight $mt^{\alpha}$. This completes the proof.
\end{proof}
As a consequence, we obtain a highest weight classification of finite-dimensional simple modules in $\BGG$; its proof is identical to that of \cite[Proposition 3.7]{M}.
\begin{cor} \label{cor: finite-dimensional simple module O}
The simple module in Theorem \ref{thm: simple modules in O} is finite-dimensional if and only if: after a rearrangement of the $\alpha_k,\beta_k$ we have $\alpha_k - \beta_k \in \BZ_{\geq 0} + \hbar^{-1} (\BZ+\BZ\tau)$ for all $1\leq k \leq n$.
\end{cor}
In \cite[Theorem 4.11]{K1} there is a similar classification in the formal setting, based on the link between quantum affine algebras and elliptic quantum groups via difference equations. See also \cite[Theorem 5.1]{C} for a different approach.
\begin{cor} \label{cor: TQ Grothendieck ring}
Let $X \in \BGG$ be finite-dimensional. In the fractional ring of $K_0(\BGG)$, $[X]$ is a Laurent polynomial in the $(\frac{[\CW^{1,u}]}{[\CW^{0,u}]})_{u\in \BC}$ whose coefficients are $\BZ$-linear combinations of the isomorphism classes of one-dimensional objects in $\BGG$.
\end{cor}
\begin{proof}
Since $[X] \in K_0(\BGG)$ is a sum of isomorphism classes of finite-dimensional simple objects, in view of Corollaries \ref{cor: tensor product factorization simple O}--\ref{cor: finite-dimensional simple module O}, we may assume $X = V^l$ where $l \in \BZ_{\geq 0}$ and $V^l$ is the simple socle of $\CW^l$; see Remark \ref{rem: simplicity asymptotic}. By Example \ref{example: q-character asymptotic representations},
\begin{align*}
\qc(V^l) &= \sum_{j=0}^l \left[\frac{\theta(z+(l+1)\hbar) \theta(z)}{\theta(z+j\hbar)}, \theta(z+(j+1)\hbar)\right] t^{l - 2j} \\
&= \sum_{j=0}^l \qc(D_j) \frac{\qc(\CW^l)}{\qc(\CW^j)} \frac{\qc(\CW^{-1})}{\qc(\CW^{j-1})}.
\end{align*}
$D_j$ is the one-dimensional $\CE$-module of highest weight $[\theta(z+(j+1)\hbar),\theta(z+(j+1)\hbar]$ for $0\leq j \leq l$. Replacing $\qc(?)$ by $[?]$ we obtain the desired result.
\end{proof}
Corollary \ref{cor: TQ Grothendieck ring} corresponds to \cite[Theorem 4.8]{FH}, and can be viewed as generalized Baxter relations in the category $\BGG$.
\section{Transfer matrices and functional relations}\label{s-TQ} 
Fix an even positive integer $L  \in 2\BZ_{>0}$ and $a_1,a_2,\cdots,a_L \in \BC\setminus (\BZ+\BZ\tau)$. Let $\BV_L$ be the $\BM$-linear subspace of $\BV^{\otimes L}$ spanned by the $v_{i_1} \otimes v_{i_2} \otimes \cdots \otimes v_{i_L} =: v_{\underline{i}}$ such that $\sum_{l=1}^L i_l = 0$; i.e., there are as many positive $i_k$s as negative $i_k$s. Following \cite{FV2} we construct an action of $K_0(\BGG)$ on $\BV_L$ by difference operators.

Introduce formal variables $p^{\alpha}$ for $\alpha \in \BC$ such that $p^{\alpha}p^{\beta} = p^{\alpha+\beta}$. For $\alpha \in \BC$, define $T_{\alpha} \in (D_{\BV_L})_{-\alpha,-\alpha}$ by $T_{\alpha}(v_{\underline{i}}) = v_{\underline{i}}$; see also \cite[\S 4.2]{EV}.

Let $X \in \BGG$. Associate to two basis vectors $v_{\underline{i}}, v_{\underline{j}} \in \BV_L$ the operator 
$$ L_{\underline{i}\underline{j}}^X(z) := L_{i_1j_1}(z+a_1-\hbar) L_{i_2j_2}(z+a_2-\hbar) \cdots L_{i_Lj_L}(z+a_L-\hbar) \in (D_X)_{0,0}. $$ 
Since $(D_X)_{0,0} \subseteq \End_{\BM}(X)$, one can take trace of $L_{\underline{i}\underline{j}}^X(z)$ over weight spaces.
\begin{defi} \label{def: transfer matrix}
The {\it transfer matrix} associated to $X \in \BGG$ is the formal sum
$$ t_X(z;p) :=  \sum_{\alpha\in \wt(X)} p^{\alpha} T_{\alpha} \sum_{\underline{i},\underline{j}} E_{\underline{i}\underline{j}} \mathrm{Tr}_{X[\alpha]}\left(L_{\underline{i}\underline{j}}^X(z)|_{X[\alpha]}\right).  $$
Here the $E_{\underline{i}\underline{j}} \in \End_{\BM}(\BV_L)$ are elementary matrices associated to the $v_{\underline{i}}$.
\end{defi}
The coefficients of the $p^{\alpha}T_{\alpha}$ are $\End_{\BM}(\BV_L)$-valued meromorphic functions on $z \in \BC$. When $X$ is finite-dimensional, one may take $p = e^{w}$ for a certain complex number $w$ and $p^{\alpha} = e^{w\alpha}$. Then $t_X(z;p) \in D_{\BV_L} \subseteq \End_{\BC}(\BV_L)$.
\begin{rem}
Consider the $\CE$-module $\BV$ constructed after Equation \eqref{rel: RLL explicit}. $t_{\BV}(z;1)$ can be identified with the transfer matrix $T(z)$ in \cite[Theorem 1]{FV2} with $W$ being the tensor product $\Psi_{a_1\hbar^{-1}}^*(\BV) \wtimes \Psi_{a_2\hbar^{-1}}^*( \BV) \wtimes \cdots \wtimes \Psi_{a_L\hbar^{-1}}^*( \BV)$.
\end{rem}
\begin{example} \label{ex: one-dim transfer}
Let $D \in \mathcal{O}$ be one-dimensional of highest weight $mt^{\alpha}$. By Theorem \ref{thm: simple modules in O} and Corollary \ref{cor: tensor product factorization simple O}, there exists $g(x) \in \BM^{\times}$ such that $mt^{\alpha} = [g(z),g(z)]t^0$.  It follows that $t_D(z;p)= \prod_{l=1}^Lg(z+a_l-\hbar)$.
\end{example}
\begin{prop} \label{prop: transfer matrices}
Let $X,Y \in \BGG$ and $\ell,u \in \BC$. The following equations hold.
\begin{itemize}
\item[(i)] $t_{\Psi_u^*X}(z;p) = t_X(z+u\hbar;p)$.
\item[(ii)] $t_X(z;p) t_Y(z;p) = t_{X\wtimes Y}(z;p)$.
\item[(iii)] $t_{\CW^{\ell}}(z;p) t_{\CW^0}(z+u\hbar;p) = t_{\CW^{\ell-u}}(z+u\hbar;p) t_{\CW^u}(z;p)$. 
\item[(iv)] $t_X(z;p)t_Y(w;p) = t_Y(w;p)t_X(z;p)$.
\end{itemize}
\end{prop}
\begin{proof}
(i) is clear from Equation \eqref{rel: spectral automorphism}. For $v_{\underline{i}}$ and $v_{\underline{j}}$ two basis vectors in $\BV_L$, 
$$L_{\underline{i}\underline{j}}^{X\wtimes Y}(z) = \sum_{k_1,k_2,\cdots,k_L\in \{\pm\}} \prod_{l=1}^L \left(L_{i_lk_l}^X(z+a_l-\hbar) \wtimes L_{k_lj_l}^Y(z+a_l-\hbar) \right).$$
The terms with $\sum_{l}k_l \neq 0$ disappear in $t_{X\wtimes Y}(z;p)$, as they have zero trace over the weight spaces.  (ii) is based on the following identity: let $\alpha \in \wt(X),\beta \in \wt(Y)$, 
$$ T_{\alpha+\beta} \mathrm{Tr}_{X[\alpha]\wtimes Y[\beta]}(f\wtimes g) = T_{\alpha} \mathrm{Tr}_{X[\alpha]}(f) T_{\beta} \mathrm{Tr}_{Y[\beta]}(g)\quad \mathrm{for}\ f \in (D_X)_{0,0},\ g \in (D_Y)_{0,0}. $$
(i), (ii) and Theorem \ref{thm: identity in K asymptotic representations} imply (iii). (iv) is proved in the same way as \cite[Theorem 5.3]{FH}, using the commutativity of $K_0(\BGG)$ in Corollary \ref{cor: q-character ring homo}.
\end{proof}
Notice that for $\ell \in \BC$, the $L_{ij}^{\CW^{\ell}}(z)$ are entire on $z$. The coefficients of $p^{\alpha}T_{\alpha}$ in $t_{\CW^{\ell}}(z;p)$ are $\End_{\BM}(\BV_L)$-valued entire functions on $z$. 
\begin{defi} \label{definition: Baxter Q operator}
The {\it Baxter Q-operator} is $Q(z;p) := t_{\CW^{z\hbar^{-1}}}(0;p)$.
\end{defi}
$t_{\CW^0}(z;p)$ is a power series in $p^{-2}$. Its constant term $t_{\CW^0}(z;\infty)$ is the scalar product by
$\prod_{l=1}^L \theta(z+a_l)$.
We have made the assumption $a_i \notin \BZ+\BZ\tau$ so that  $t_{\CW^0}(0;p) = Q(0;p)$ is an invertible power series in $p^{-2}$. Furthermore, let us write $Q(z;p) = p^{z\hbar^{-1}} T_{z\hbar^{-1}} \widetilde{Q}(z;p)$. Then $\widetilde{Q}(z;p) = \sum_{k=0}^{\infty} p^{-2k} T_{-2k} \widetilde{Q}_k(z)$ and $\widetilde{Q}_0(0) = Q(0;\infty) = \prod_{l=1}^L \theta(a_l)$. The power series $\widetilde{Q}(z;p)$ can therefore be inverted, with coefficients of $p^{-2k}T_{-2k}$ being $\End_{\BM}(\BV_L)$-valued entire functions on $z$,  by Proposition \ref{prop: finit-dim representations}. Now we arrive at the second main result of this paper.
\begin{theorem}  \label{thm: Baxter}
\begin{itemize}
\item[(i)] $\frac{t_{\CW^{\ell}}(z;p)}{t_{\CW^0}(z;p)} = \frac{Q(z+\ell\hbar;p)}{Q(z;p)}$ for $\ell \in \BC$.
\item[(ii)] Assume $a_1=a_2 = \cdots = a_L = a$ and $L = 2n$. Then 
$$ \widetilde{Q}(z+1;p) = (-1)^n \widetilde{Q}(z;p),\quad \widetilde{Q}(z+\tau;p) = (-1)^n e^{-n\BI\pi(\tau + 2z+2a)} \widetilde{Q}(z;p)  $$
where $\widetilde{Q}(z;p) := p^{-z\hbar^{-1}} T_{-z\hbar^{-1}} Q(z;p)$.
\end{itemize} 
\end{theorem}
\begin{proof}
The formal power series $t_{\CW^0}(z;p)$ and $Q(z;p)$ are both invertible. Replacing $(\ell,u,z)$ in Proposition \ref{prop: transfer matrices} (iii) by $(z\hbar^{-1}+\ell,z\hbar^{-1},0)$ we obtain
$$ Q(z+\ell \hbar;p) t_{\CW^0}(z;p) = t_{\CW^{\ell}}(z;p) Q(z;p), $$
which leads to (i) by the commutativity of transfer matrices.

Consider the representation $\CW^{\ell}$ of $\CE$ on the same underlying $\BM$-vector space $\CW := \oplus_{k=0}^{\infty} \BM w_k$ in Proposition \ref{prop: finit-dim representations}. The second row of $L^{\ell}(z;x)$ is independent of $\ell$, while the first row are $\End_{\BM}(\CW)$-valued entire functions $g(\ell)$ on $\ell$ satisfying $g(\ell) \sim \theta(\ell \hbar + z +\hbar)$.
Under the assumption $a_l = a$ for $1\leq l \leq L$, the $L_{\underline{i}\underline{j}}^{\CW^{\ell \hbar^{-1}}}(0)$, as $\End_{\BM}(\CW)$-valued entire functions $h(\ell)$ on $\ell$, satisfy $h(\ell) \sim \theta(\ell+a)^n$
because the number of $1\leq l \leq 2n$ with $i_l = +$ is $n$. From 
$$\widetilde{Q}(z;p) = \sum_{k=0}^{\infty} p^{-2k} T_{-2k} \sum_{\underline{i},\underline{j}} E_{\underline{i}\underline{j}}\mathrm{Tr}_{\BM w_k}\left(L_{\underline{i}\underline{j}}^{\CW^{z\hbar^{-1}}}(0)\right) $$
we obtain the desired double periodicity in (ii).
\end{proof}
\begin{rem}
In Theorem \ref{thm: Baxter}, (i) is to be compared with \cite[Corollary 4.6]{FH}, and (ii) with \cite[Theorem 5.9]{FH}. Combining the proof of Corollary \ref{cor: TQ Grothendieck ring} with Theorem \ref{thm: Baxter} (i), we conclude as in \cite[Theorem 4.8]{FH} that for $n \in \BZ_{\geq 0}$:
\begin{equation} \label{TQ}
t_{V^n}(z;p) = \sum_{j=0}^n \frac{Q(z+n\hbar;p)Q(z-\hbar;p)}{Q(z+j\hbar;p)Q(z+(j-1)\hbar;p)} \prod_{l=1}^L \theta(z+a_l+j\hbar).
\end{equation}
\end{rem}
\begin{rem}
 It is important to understand the convergence of $\widetilde{Q}(z;p)$ with respect to $p$, as in \cite[Remark 5.12(ii)]{FH}. For example, let $L = 2$. With respect to the basis $(v_+\otimes v_-,\ v_-\otimes v_+)$ of $\BV_2$ is the two-by-two matrix $\widetilde{Q}(z;p) = \begin{pmatrix}
A & B \\
C & D
\end{pmatrix}$:
\begin{align*}
A&= \sum_{j=0}^{\infty} p^{-2j} T_{-2j} \frac{\theta(z+a_1-j\hbar)\theta(z+x+(-j+1)\hbar)\theta(x-j\hbar)}{\theta(x)\theta(z+x+(-2j+1)\hbar)} \theta(a_2+j\hbar),  \\
B&= -\sum_{j=1}^{\infty} p^{-2j}  T_{-2j}\frac{\theta(z+a_1+x-j\hbar)\theta(z-(j-1)\hbar)}{\theta(z+x+(-2j+1)\hbar)} \frac{\theta(a_2-x+j\hbar)\theta(j\hbar)}{\theta(x-\hbar)}, \\
C&=  -\sum_{j=0}^{\infty} p^{-2j}T_{-2j} \frac{\theta(a_1-x+j\hbar)\theta((j+1)\hbar)}{\theta(x)}  \frac{\theta(z+a_2+x-j\hbar)\theta(z-j\hbar)}{\theta(z+x-2j\hbar)}, \\
D&= \sum_{j=0}^{\infty} p^{-2j} T_{-2j} \theta(a_1+j\hbar) \frac{\theta(z+a_2-j\hbar)\theta(z+x-j\hbar)\theta(x-(j+1)\hbar)}{\theta(x-\hbar)\theta(z+x-2j\hbar)}.
\end{align*}
The convergence of $A,B,C,D$ in $p$ is still unclear to us.
\end{rem}

\begin{rem} \label{rem: BAE}
Assume $a_1=a_2=\cdots = a_L = a$ and $L = 2n$. Suppose that $\widetilde{Q}(z;p)$ converges for certain $p = e^w$. If $Q(z;p)$ has an eigenvalue $q(z)$ which is independent of $x$, then by Theorem \ref{thm: Baxter}(ii), $q(z) = c p^{z\hbar^{-1}} \prod_{i=1}^n \theta(z-z_i)$ where $c \in \BC^{\times}$ and $z_i\in \BC$. Since $t_{V^1}(z;p)$ is entire on $z$, by Equation \eqref{TQ} we have: 
\begin{eqnarray}
&& p^2 \left(\frac{\theta(z_k + a)}{\theta(z_k+a+\hbar)}\right)^{2n} =  \prod_{1\leq j \leq n,j\neq k} \frac{\theta(z_k-z_j-\hbar)}{\theta(z_k-z_j+\hbar)}\quad \mathrm{for}\ 1\leq k \leq n, \label{equ: Bethe equations} \\
&& z_1 + z_2 + \cdots + z_n - n a \in \BZ + \BZ\tau. \label{equ: Baxter sum rule}
\end{eqnarray}
The first equation was deduced in \cite{FV2} from Bethe ansatz, and the second is Baxter's sum rule for homogeneous models ($a_1 = a_2 = \cdots = a_L$), to be compared with \cite{Takebe} for the eight-vertex model based on representations of the Sklyanin's elliptic algebra (the elliptic quantum group associated with $\mathfrak{sl}_2$ of vertex type \cite{twist}). 
\end{rem}
\section{Further discussions}
In this paper, for the elliptic quantum group associated to $\mathfrak{sl}_2$, we have introduced a category $\BGG$ of representations and studied in details the asymptotic representations $\CW^{\ell,u}$. The $\CW^{\ell,u}$ satisfy a \lq\lq separation of variables\rq\rq\ identity (Theorem \ref{thm: identity in K asymptotic representations}) and are used to construct the elliptic Baxter Q-operator (Definition \ref{definition: Baxter Q operator}).

For an affine quantum group (of an arbitrary complex finite-dimensional simple Lie algebra), Baxter Q-operators have been constructed in \cite[\S 4]{FH} as transfer matrices over {\it positive pre-fundamental representations}, generalizing previous works of Bazhanov--Lukyanov--Zamolodchikov on $\mathfrak{sl}_2$ \cite{BazhanovLukyanovZamolodchikov1997,
BazhanovLukyanovZamolodchikov1999}. The pre-fundamental representations can only be defined over a Borel subalgebra instead of the full quantum group \cite{HJ}. Let us comment that in the elliptic case the Borel subalgebras and their pre-fundamental representations are still unavailable. 

The ideas of the paper should work for affine quantum groups and Yangians, leading to a second construction of Q-operators without reference to Borel subalgebras \cite{HJ,FH} or double Yangians \cite{B2}. As an illustrating example, in the appendix we discuss the two definitions of Q-operators for the Yangian of $\mathfrak{sl}_2$. In general, asymptotic representations should be in corresponding categories $\BGG$ \cite[\S 3]{GTL} (non-integrable); see \cite[\S 5]{Z2} in the case of quantum affine $\mathfrak{gl}(m|n)$.

We plan to work on elliptic quantum groups for Lie (super)algebras of higher ranks \cite{C,GS,K2}. The asymptotic representations might be deduced from those of affine quantum groups by twistors \cite{twist} or by difference equations \cite{GTL,K2}. They might also be realized directly as \lq\lq asymptotic limits\rq\rq\ of finite-dimensional representations \cite{HJ,Z2}; see a particular example for $\CE_{\tau,\hbar}(\mathfrak{gl}_N)$ in \cite[\S 3.4]{C}.

An important property of Q-operators is that their roots satisfy the Bethe ansatz equations. For elliptic quantum $\mathfrak{sl}_2$, this is derived from the two-dimensional irreducible representations; see Remark \ref{rem: BAE}. In higher ranks, such two-dimensional representations do not exist, and we should consider infinite-dimensional representations, as indicated in the recent works on affine quantum groups \cite{HL,FH2,Jimbo2} and quantum toroidal $\mathfrak{gl}_1$ \cite{Jimbo1}. 

At last let us mention another work \cite{FS} on quantum integrable systems with boundary conditions. In {\it loc. cit.}, the Q-operator is the Sklyanin transfer matrix of the pre-fundamental representation of the Yangian of $\mathfrak{sl}_2$ (see Example \ref{Y: prefund}), but the proof of TQ relations is less representation-theoretic. It remains open what a tensor category $\BGG$ of representations would be in this context.

\medskip

\noindent {\bf Acknowledgments.} We thank Niklas Beisert, David Hernandez, Michio Jimbo and Alexander Varchenko for stimulating discussions. The research of the authors is supported by the National Center of Competence in Research SwissMAP---The Mathematics of Physics of the Swiss National Science Foundation.

\appendix
\section{Baxter operators for the Yangian}
In this appendix we study the asymptotic representations of the Yangian algebra $\BY(\mathfrak{gl}_2)$ and use their spin parameter to define the Baxter Q-operator. Then we compare our Q-operator with that in \cite{B}; see Remark \ref{rem: two Qs}. The results (category $\BGG$, $q$-characters, functional relations of transfer matrices, etc.) are almost identical to the elliptic case. Most of their proofs will be omitted.

Throughout this section, vector spaces, linear maps and tensor products are over $\BC$.  
Let us first define the Yangian. Fix a basis $(v_1,v_2)$ of the vector space $\BC^2$. With respect to the basis $(v_1^{\otimes 2}, v_1\otimes v_2, v_2\otimes v_1, v_2^{\otimes 2})$ of $\BC^2\otimes \BC^2$ set
$$R(z) := \begin{pmatrix}
1 & 0 & 0 & 0 \\
0 & \frac{z}{z+1} & \frac{1}{z+1} & 0 \\
0 & \frac{1}{z+1} & \frac{z}{z+1} & 0 \\
0 & 0 & 0 & 1
\end{pmatrix} \in \End(\BC^2)^{\otimes 2}. $$
It is well-known that $R$ verifies the quantum Yang--Baxter equation in $\End(\BC^{2})^{\otimes 3}$:
$$ R^{12}(z-w)R^{13}(z)R^{23}(w) = R^{23}(w)R^{13}(z)R^{12}(z-w). $$
Let $r \in \BZ_{\geq 0}$. The {\it Yangian} $\BY^r$  is an algebra generated by $t_{ij}^{(n)}$ with $1\leq i,j \leq 2$ and $n \in \BZ_{\leq r}$. Introduce generating functions $T(z) \in \End(\BC^2) \otimes \BY^{r}((z^{-1}))$ \footnote{The ordinary Yangian $\BY(\mathfrak{gl}_2)$ is the quotient of $\BY^0$ by $t_{ij}^{(0)} = \delta_{ij}$.}
$$ T(z) := \sum_{1\leq i,j \leq 2} E_{ij} \otimes t_{ij}(z),\quad t_{ij}(z) := \sum_{n\leq r} t_{ij}^{(n)}z^n $$
where the $E_{ij} \in \End (\BC^2)$ are elementary matrices for $1\leq i,j \leq 2$. The defining relation of $\BY^{r}$ is the following identity in $\End(\BC^2)^{\otimes 2} \otimes \BY^r((z^{-1},w^{-1}))$: 
\begin{equation}  \label{equ: RTT}
R^{12}(z-w) T^{13}(z)T^{23}(w) = T^{23}(w)T^{13}(z)R^{12}(z-w).
\end{equation}
One has natural projections of algebras $\BY^{r+1} \longrightarrow \BY^r: t_{ij}^{(n)} \mapsto t_{ij}^{(n)}$. To $r,s \in \BZ_{\geq 0}$ is attached an algebra homomorphism 
 \begin{equation}  \label{equ: coproduct}
 \Delta^{rs}: \BY^{r+s} \longrightarrow \BY^r \otimes \BY^s,\quad  T(z) \mapsto T_{12}(z)T_{13}(z).
\end{equation}  
The $\Delta^{rs}$ are co-associative: $(1\otimes \Delta^{rs})\Delta^{r,s+q} = (\Delta^{rs}\otimes 1)\Delta^{r+s,q}$. There is a one-parameter family $(\Psi_u: \BY^r \longrightarrow \BY^r)_{u\in \BC}$ of algebra automorphisms defined by  
$$\Psi_u:\quad t_{ij}(z) \mapsto t_{ij}(z+u) \in \BY^r((z^{-1})). $$
The algebra $\BY^r$ is $\BZ$-graded with: $|t_{ij}^{(n)}|_{\BZ} = j-i$. A {\it graded representation} of $\BY^r$ is a $\BZ$-graded vector space $V = \oplus_{\alpha\in \BZ} V_{\alpha}$ (weight space decomposition) endowed with a morphism of graded algebras $\rho: \BY^r \longrightarrow \End (V)$. Call $V$ a graded $\BY^r$-module, and let $\wt(V)$ denote the set of weights of $V$ as in Section~\ref{ss-h algebras}. Write $t_{ij}^V(z) = \rho(t_{ij}(z)) \in \End (V)((z^{-1}))$ and $T^V(z) = (1\otimes \rho)(T(z))$.

If $(\rho,V)$ and $(\sigma,W)$ are graded representations of $\BY^r$ and of $\BY^s$ respectively, then $(\rho \otimes \sigma) \circ \Delta^{rs}$ makes $V\otimes W$ into a graded representation of $\BY^{r+s}$. 

\begin{example}  \label{Y: fin-dim}
Let $m \in \BZ_{>0}$. Let $V^m$ be the vector space with basis $(w_i)_{0\leq i \leq m}$. Set $w_i$ to be of weight $i$. There is a graded $\BY^1$-module structure on $V^m$:
 \begin{gather*}
 t_{11}(z) w_i = (z+m-i) w_i,\quad t_{12}(z) w_i = (m-i)w_{i+1}, \\
 t_{21}(z) w_i = i w_{i-1}, \quad t_{22}(z) w_i = (z+i) w_i.
 \end{gather*}
 Here $w_{-1} = w_{m+1} = 0$ by convention. This comes from the evaluation map
 $$ \mathrm{ev}: \BY^1 \mapsto U(\mathfrak{gl}_2),\quad t_{ij}(z) \mapsto \delta_{ij} z - E_{ij}, $$
 where $U(\mathfrak{gl}_2)$ denotes the universal enveloping algebra of $\mathfrak{gl}_2$; see e.g. \cite[(2.5)]{M}.
\end{example}
\begin{example} \label{Y: asym}
 Let $V^{\infty}$ be the vector space with basis $(w_i: i \in \BZ_{\geq 0})$ and with $w_i$ being of weight $i$. Fix $\ell \in \BC$. In the formulas of the $t_{ij}(z)w_k$ in Example \ref{Y: fin-dim}, one can replace $m$ by $\ell$ everywhere, leading to a graded $\BY^1$-module structure on $V^{\infty}$, denoted by $\CW^{\ell}$. If $\ell \in \BZ_{>0}$, then $V^{\ell}$ is the simple socle of $\CW^{\ell}$ under the natural embedding $V^{\ell} \subset V^{\infty}, w_i \mapsto w_i$. 
 
For $\ell, u \in \BC$, define the {\it asymptotic representation} $\CW^{\ell,u} := \Psi_u^*(\CW^{\ell})$. 
\end{example}

\begin{example} \label{Y: prefund}
There is another graded $\BY^1$-module structure on $V^{\infty}$:
$$ t_{11}(z) w_i = (z-i)w_i,\quad t_{12}(z)w_i = -w_{i+1},\quad t_{21}(z)w_i = iw_{i-1},\quad t_{22}(z)w_i = w_i. $$
This module, denoted by $\BW$, was constructed in \cite[(3.37)]{B} as a Fock space representation of the harmonic oscillator algebra.\footnote{Tensoring $\CW^{\ell,u}$ with the one-dimensional module $t_{ij}(z) = \delta_{ij}z^{-1}$ in $\BGG^0$ results in a module over $\BY(\mathfrak{gl}_2)$. It is unknown how to transform $\BW$ into a $\BY(\mathfrak{gl}_2)$-module.}
\end{example}

A graded $\BY^{r}$-module $V$ is said to be in category $\BGG^r$ if:
 \begin{itemize}
 \item[(i)] $\wt(V)$ is bounded below and all the weight spaces are finite-dimensional;
 \item[(ii)] $T^V(z)\in \End (\BC^2\otimes V)((z^{-1}))$ and $t_{22}^V(z) \in \End (V)((z^{-1}))$ are invertible as matrix-valued Laurent series.
 \end{itemize}
The projections $\BY^{r+1} \longrightarrow \BY^r$ imply $\BGG^0 \subseteq \BGG^1 \subseteq \BGG^2 \subseteq \cdots$. 
 
 By condition (ii), the two-by-two matrix $T^V(z)$ admits a Gauss decomposition:
\begin{equation*}  
T^V(z) = \begin{pmatrix}
1 & f^V(z) \\
0 & 1 
\end{pmatrix} \begin{pmatrix}
K_1^V(z) & 0 \\
0 & K_2^V(z) 
\end{pmatrix} \begin{pmatrix}
1 & 0 \\
e^V(z) & 1
\end{pmatrix}.
\end{equation*}
The $K_i^V(z)\in \End (V)((z^{-1}))$ are invertible and $K_i^V(z)K_j^V(w) = K_j^V(w)K_i^V(z)$. For $\alpha \in \wt(V)$, the commuting family of linear operators $K_{i}^V(z)|_{V_{\alpha}}$ decomposes the finite-dimensional vector space $V_{\alpha}$ into a direct sum of the generalized eigenspaces $V_{(g_1,g_2;\alpha)}$ where $g_1(z),g_2(z) \in \BC((z^{-1}))^{\times}$ and
$$ V_{(g_1,g_2;\alpha)} := \{ v \in V_{\alpha} \ |\ (K_{i}^V(z)-g_{i}(z))^N v = 0\quad \mathrm{for}\ N > \dim V_{\alpha},\ i = 1,2 \}. $$
Let $\mathcal{R}$ be the group ring over $\BZ$ of the multiplicative group $\BC((z^{-1}))^{\times} \times \BC((z^{-1}))^{\times}$; elements in $\mathcal{R}$ are $\BZ$-linear combinations of the $[g_1,g_2]$ with $g_i \in \BC((z^{-1}))^{\times}$. Introduce a formal variable $p$. The $q$-character of $V$ is a Laurent series in $p$ with coefficients in $\mathcal{R}$; see \cite[\S 2]{Kn} or \cite[\S 7.4]{GTL} with $p = 1$,
\begin{equation*} 
\qcy(V) := \sum_{\alpha \in \wt(V)} p^{\alpha} \sum_{g_1,g_2 \in \BC((z^{-1}))^{\times}} [g_1,g_2] \dim V_{(g_1,g_2;\alpha)} \in \mathcal{R}((p)).
\end{equation*}
Notice that $\qcy(V)$ is independent of the choice of $r$ such that $V \in \BGG^r$. 

Let $K_0(\BGG^r)$ denote the completed Grothendieck group of the abelian category $\BGG^r$. The $q$-character map induces an injective morphism of additive groups $\qcy: K_0(\BGG^r) \longrightarrow \mathcal{R}((p))$. Furthermore, for $V \in \BGG^r$ and $W \in \BGG^{s}$ we have
$$ V \otimes W \in \BGG^{r+s}\quad \mathrm{and}\quad \qcy(V\otimes W) = \qcy(V) \qcy(W).$$
\begin{rem} \label{rem: Y q-char}
The $\BY^1$-modules $V^m, \CW^{\ell}$ and $\BW$ in Examples \ref{Y: fin-dim}--\ref{Y: prefund} are in $\BGG^1$:
\begin{align*}
\qcy(V^m) &=  [z+m,z] \sum_{i=0}^{m} p^{i}\left[\frac{z-1}{z+i-1}, \frac{z+i}{z}\right], \\
\qcy(\CW^{\ell}) &= [z+\ell,z] \sum_{i=0}^{\infty} p^{i}\left[\frac{z-1}{z+i-1}, \frac{z+i}{z}\right] = [z+\ell,z] \qcy(\CW^0), \\
 \qcy(\BW) &= \sum_{i=0}^{\infty} p^{i} [z,1] = \frac{[z,1]}{1-p}. 
\end{align*}
As a consequence of the injectivity of $q$-characters, the same identity as in Theorem \ref{thm: identity in K asymptotic representations} holds: $[\CW^{\ell,0} \otimes \CW^{0,u}] = [\CW^{\ell-u,u} \otimes \CW^{u,0}]$ in $K_0(\BGG^2)$.
\end{rem}

 Fix a positive integer $L \in \BZ_{>0}$ and $a_1,a_2,\cdots,a_L \in \BC^{\times}$. Define the {\it quantum space} $V_L :=  (\BC^2)^{\otimes L}$. We identify 12-strings $i_1i_2\cdots i_L$ of length $L$ with the vector $v_{i_1}\otimes v_{i_2} \otimes \cdots \otimes v_{i_L} \in V_L$. For $0 \leq s \leq L$, let $V_L^s$ be the subspace of $V_L$ spanned by the $i_1i_2\cdots i_L$ where $1$ appears exactly $s$ times.

Let $W \in \BGG^{r}$. Define the {\it monodromy matrix} in $\End(V_L\otimes W)((z^{-1}))$
$$ T^{W,L}(z) := T^W(z+a_1)_{1,L+1} T^W(z+a_2)_{2,L+1} \cdots T^W(z+a_L)_{L,L+1}. $$
The {\it transfer matrix} is obtained by taking twisted trace over $W$:
$$ t_W(z;p) := \sum_{\alpha \in \wt(W)} p^{\alpha} (\mathrm{Id}_{V_L} \otimes \mathrm{Tr}_{W_{\alpha}})(T^{W,L}(z)) \in \End(V_L)((z^{-1},p)). $$
Proposition \ref{prop: transfer matrices} holds by setting $\hbar = 1$. In particular, the transfer matrices are commuting operators in $V_L$. From the weight grading we observe that $V_L^s$ is stable by $t_W(z;p)$ for $0\leq s \leq L$. Notice that $t_{ij}^{\CW^{\ell}}(z) \in \End(V^{\infty})[z,\ell]$ in Examples \ref{Y: fin-dim}--\ref{Y: asym}. The following definition makes sense.
\begin{defi} \label{def: Y Baxter}
The Baxter Q-operator is defined to be $Q(z;p) := t_{\CW^z}(0;p)$.
\end{defi}
\begin{theorem} \label{thm: Q Yangian}
For $0\leq s \leq L$, the $\End(V_L^s)[[p]]$-valued polynomial $Q(z;p)|_{V_L^s}$ in $z$ is of degree $s$. 
\end{theorem}
\begin{proof}
In the two-by-two matrix $T^{\CW^{\ell}}(z)$, the second row is independent of $\ell$, and the first row is a polynomial in $\ell$ of degree 1. So in $$T^{\CW^{z},L}(0) = \sum_{\underline{i}\underline{j}} E_{i_1j_1}\otimes E_{i_2j_2} \otimes  \cdots \otimes E_{i_Lj_L} \otimes t_{i_1j_1}^{\CW^{z}}(a_1) t_{i_2j_2}^{\CW^z}(a_2) \cdots t_{i_Lj_L}^{\CW^z}(a_L) $$
 only the terms $t_{i_lj_l}^{\CW^{z}}(a_l)$ with $i_l = 1$ raise possibly the power of $z$ by 1. In $V_L^s$, the number of such $i_l$ is $s$. So $Q(z;p)|_{V_L^s}$ is a polynomial in $z$ of degree $\leq s$. We follow the idea of \cite[\S 5.3]{FH} to prove that $Q(z;0)|_{V_L^s}$ is of degree $s$. 

Let us order the basis $(i_1i_2\cdots i_L)$ of $V_L$ as follows: $i_1i_2\cdots i_L \prec j_1j_2\cdots j_L$ if there exists $1\leq t \leq L$ such that $i_t = 1, j_t = 2$ and $i_l = j_l$ for all $l > t$. This is a total ordering. Furthermore, notice that for $\underline{j} \prec \underline{i}$, the last tensor factor in the above summation annihilates $w_0$. This means that $Q(z;0)$ is upper triangular with respect to this ordering. Its diagonal term associated to $i_1i_2\cdots i_L \in V_L^s$ is $\prod_{l=1}^L (a_l + \delta_{i_l1} z)$, which is a polynomial in $z$ of degree $s$ as $a_l \neq 0$ for $1\leq l \leq L$.
\end{proof}
Theorem \ref{thm: Baxter} (i) still holds by taking $\hbar = 1$. From Remark \ref{rem: Y q-char} we deduce functional relations between the $t_{V^m}(z;p)$ and $Q(z;p)$:
\begin{align*}
\qcy(V^m) &= \sum_{i=0}^mp^i[z+i,z+i]\frac{\qcy(\CW^m) \qcy(\CW^{-1})}{\qcy(\CW^i)\qcy(\CW^{i-1})} \quad \mathrm{for}\ m \in \BZ_{\geq 0}.
\end{align*}
In terms of transfer matrices the case $m=1$ leads to the Baxter TQ relation 
$$ t_{V^1}(z;p) = \frac{Q(z+1;p)}{Q(z;p)} \prod_{l=1}^L (z+a_l) + \frac{Q(z-1;p)}{Q(z;p)} p \prod_{l=1}^L (z+a_l+1). $$
The convergence of $Q(z;p)$ with respect to $p$ is much simpler than in the elliptic case. Indeed, the matrix coefficients of $Q(z;p)$ with respect to the basis $(\underline{i})$ are linear combinations of the power series $\sum_{i=0}^{\infty} p^i i ^s$ with $0\leq s \leq L$; these are rational functions in $p$ whose only possible p\^{o}le is at $p = 1$.
\begin{rem} \label{rem: two Qs}
Let us compare $Q(z;p)$ with the Q-operator $\mathbf{Q}_+(z)$ in \cite[(3.51)]{B}. Indeed, $\mathbf{Q}_+(z)$ can be identified with the transfer matrix $t_{\BW}(z;p)|_{p=e^{\sqrt{-1}\phi}}$ associated to $\BW$ in Example \ref{Y: prefund}. By Remark \ref{rem: Y q-char}, $\frac{[\CW^{\ell}]}{[\CW^0]} = \frac{[\Psi_{\ell}^*(\BW)]}{[\BW]}$. It follows that
$$ \frac{Q(z+\ell;p)}{Q(z;p)} = \frac{t_{\CW^{\ell}}(z;p)}{t_{\CW^0}(z;p)} = \frac{t_{\BW}(z+\ell;p)}{t_{\BW}(z;p)}\quad \mathrm{for}\ \ell \in \BC. $$
For $0\leq s \leq L$, $t_{\BW}(z;p)|_{V_L^s}$ is an $\End(V_L^s)[[p]]$-valued polynomial in $z$ of degree $z^s$, the coefficient of $z^s$ being $\frac{1}{1-p}$. Let $A_L^s(p) \in \End(V_L^s)[[p]]$ be the coefficient of $z^s$ in $Q(z;p)|_{V_L^s}$. The above equality implies  
$$Q(z;p)|_{V_L^s} = (1-p)A_L^s(p) \times t_{\BW}(z;p)|_{V_L^s} \quad \mathrm{for}\ 0 \leq s \leq L. $$ 
\end{rem}
\begin{rem}
From the proof of Theorem \ref{thm: Q Yangian}  we see that: $A_L^s(0)$ is upper triangular whose diagonal associated to $\underline{i}$ is $\prod_{l: i_l=2} a_{i_l}$. Under genericity conditions on the $a_l, p$, one can assume that $A_L^s(p)$ is diagonalizable and its eigenvalues are all of multiplicity one. Then $Q(z;p)|_{V_L^s}$ and the $t_W(z;p)|_{V_L^s}$ are all diagonalizable. For example, take $L = 2$ and $s = 1$. With respect to the basis $(21,12)$ of $V_2^1$, 
$$ A_2^1(p) = \frac{1}{1-p} \begin{pmatrix}
a_1 + \frac{p}{1-p} & \frac{1}{1-p} \\
\frac{p}{1-p} & a_2 + \frac{p}{1-p}
\end{pmatrix}. $$
It has two distinct eigenvalues if and only if $(a_1-a_2)^2 + \frac{4p}{(1-p)^2} \neq 0$, if and only if the following Bethe ansatz equation in $z$ has two distinct solutions:
$$p(z+a_1+1)(z+a_2+1) = (z+a_1)(z+a_2).$$
Let $z = z_1\in \BC$ be one of its solution. Then 
$$ v(z_1) := (z_1+a_1+1)v_2\otimes v_1 + (z_1+a_2) v_1\otimes v_2 $$
is a common eigenvector of $A_2^1(p), Q(z;p)$ of eigenvalues respectively
$$\lambda(z_1) := \frac{a_1}{1-p}+\frac{p}{(1-p)^2} + \frac{1}{(1-p)^2} \frac{z_1+a_2}{z_1+a_1+1},\quad \lambda(z_1) \times (z-z_1). $$ 
View $V_2$ as the graded $\BY^2$-module $\Psi_{a_1}^*(V^1) \otimes \Psi_{a_2}^*(V^1)$ by $w_i\otimes w_j = v_{i+1}\otimes v_{j+1}$. Then $v(z_1) = t_{21}(z_1) v_2^{\otimes 2}$ comes from Algebraic Bethe Ansatz \cite[\S 4]{Faddeev}. 
\end{rem} 

\end{document}